\documentclass{IEEEtran4PSCC}
\ifCLASSINFOpdf
   \usepackage[pdftex]{graphicx}
\else
   \usepackage[dvips]{graphicx}
\fi
\usepackage[cmex10]{amsmath}
\usepackage{amsthm}

\usepackage{graphicx, bm, bbm, latexsym, amsfonts, epsfig, verbatim, epsfig, verbatim, enumerate, multirow, caption, graphicx, subcaption, algorithm, algpseudocode, hyperref, url, multicol, color, latexsym, amsfonts, epsfig, verbatim, tikz, cases, algcompatible, diagbox,multicol,fixmath,makecell,booktabs}
\hypersetup{
	colorlinks=true, 
	breaklinks=true,
	urlcolor= blue, 
	linkcolor= blue, 
	citecolor=red, 
	pdftitle={},
	pdfauthor={},
}
\allowdisplaybreaks

\newcommand{\qb}{\bm{q}}
\newcommand{\cb}{\bm{c}}
\newcommand{\zb}{\bm{z}}
\newcommand{\xb}{\bm{x}}

\newcommand{\Ab}{\bm{A}}
\newcommand{\Bb}{\bm{B}}
\newcommand{\db}{\bm{d}}

\newcommand{\mub}{\bm{\mu}}

\newcommand{\betab}{\bm{\beta}}
\newcommand{\tilxi}{\bm{\tilde{\xi}}}
\newcommand{\tilxiN}{\tilde{\xi}_{\text{\tiny N}}}
\newcommand{\tilxiE}{\tilde{\xi}_{\text{\tiny E}}}

\newcommand{\Eb}{\mathbb{E}}
\newcommand{\Pb}{\mathbb{P}}

\newcommand{\zzg}{z^{\text{\tiny g}}} 
\newcommand{\zza}{z^{\text{\tiny a}}} 

\newcommand{\ffp}{f^{\text{\tiny p}}} 
\newcommand{\ffq}{f^{\text{\tiny q}}} 
\newcommand{\UBgp}{\overline{g}^{\text{\tiny p}}} 
\newcommand{\UBgq}{\overline{g}^{\text{\tiny q}}} 
\newcommand{\LBgp}{\underline{g}^{\text{\tiny p}}} 
\newcommand{\LBgq}{\underline{g}^{\text{\tiny q}}} 

\newcommand{\ddp}{d^{\text{\tiny p}}} 
\newcommand{\ddq}{d^{\text{\tiny q}}} 

\newcommand{\lpp}{l^{\text{\tiny p+}}} 
\newcommand{\lpm}{l^{\text{\tiny p--}}} 
\newcommand{\lqp}{l^{\text{\tiny q+}}} 
\newcommand{\lqm}{l^{\text{\tiny q--}}} 

\newtheorem{theorem}{Theorem}[section]
\newtheorem{remark}[theorem]{Remark}

\newtheorem{proposition}[theorem]{Proposition}

\makeatletter
\let\old@ps@headings\ps@headings
\let\old@ps@IEEEtitlepagestyle\ps@IEEEtitlepagestyle
\def\psccfooter#1{%
    \def\ps@headings{%
        \old@ps@headings%
        \def\@oddfoot{\strut\hfill#1\hfill\strut}%
        \def\@evenfoot{\strut\hfill#1\hfill\strut}%
    }%
    \def\ps@IEEEtitlepagestyle{%
        \old@ps@IEEEtitlepagestyle%
        \def\@oddfoot{\strut\hfill#1\hfill\strut}%
        \def\@evenfoot{\strut\hfill#1\hfill\strut}%
    }%
    \ps@headings%
}
\makeatother

\begin{document}
\title{\Huge Algorithms for Mitigating the Effect of Uncertain Geomagnetic Disturbances in Electric Grids}

%% To specify the authors when (number of affiliations <= 2)
\author{
\IEEEauthorblockN{Minseok Ryu$^*$, Harsha Nagarajan$^\dag$, Russell Bent$^\dag$}
\IEEEauthorblockA{$^*$Department of Industrial \& Operations Engineering, University of Michigan, Ann Arbor, \\
$^\dag$Theoretical Division (T-5), Los Alamos National Laboratory, NM, U.S.A. (Contact: harsha@lanl.gov)}
}

\maketitle

% \section{Conference paper preparation}
% Papers accepted for PSCC are normally no longer than 7 pages. However, in exceptional cases, up to two-3 additional pages can be accepted if the extra length is required to fully but still succinctly develop and explain an idea and its application. In such cases, the authors will be required to briefly explain why the extra length is necessary. 

% Please use automatic hyphenation and check your spelling. Additionally, be sure your sentences are complete and that there is continuity within your paragraphs. Check the numbering of your graphics (figures and tables) and make sure that all appropriate references are included.

\begin{abstract}
Geomagnetic disturbances (GMDs), a result of space weather, pose a severe risk to electric grids. When GMDs occur, they can cause geomagnetically-induced currents (GICs), which saturate transformers, induce hot-spot heating, and increase reactive power losses in the transmission grid. Furthermore, uncertainty in the magnitude and orientation of the geo-electric field, and insufficient historical data make the problem of mitigating the effects of uncertain GMDs challenging. In this paper, we propose a novel distributionally robust optimization (DRO) approach that models uncertain GMDs and mitigates the effects of GICs on electric grids. This is achieved via a set of mitigation actions (e.g., line switching, locating blocking devices, generator re-dispatch and load shedding), prior to the GMD event, such that the worst-case expectation of the system cost is minimized. To this end, we develop a column-and-constraint generation algorithm that solves a sequence of mixed-integer second-order conic programs to handle the underlying convex support set of the uncertain GMDs. Also, we present a monolithic exact reformulation of our DRO model when the underlying support set can be approximated by a polytope with \textit{three} extreme points. Numerical experiments on `epri-21' system show the efficacy of the proposed algorithms and the exact reformulation of our DRO model. 
\end{abstract}

\begin{IEEEkeywords}
Algorithms, Distributionally robust optimization, Geomagnetic disturbances, Line switching.
\end{IEEEkeywords}

% Use this to place sponsorships
% \thanksto{\noindent Submitted to the 21st Power Systems Computation Conference (PSCC 2020).}

\section{Introduction}
\label{sec:intro}
% [1] Background
Geomagnetic disturbances (GMDs) are caused by solar storms. During these storms charged particles escape from the sun, travel to the earth, create a geomagnetically-induced current (GIC), and impact electric transmission grids. Detrimental impacts include current distortions (harmonics), saturation of transformers, induced hot-spot heating, and increased reactive power losses \cite{gao_gmd_2018, boteler1998effects, gannon2019geomagnetically} (see Figure \ref{fig:GIC_illustration}). In 1989, the Hydro-Quebec power system was shut down for 9 hours due to a rare, but high-impact GMD, which led to a net loss of \$13.2 million \cite{bolduc2002gic}. 

\begin{figure}[h]
    \centering
    \includegraphics[width=\columnwidth]{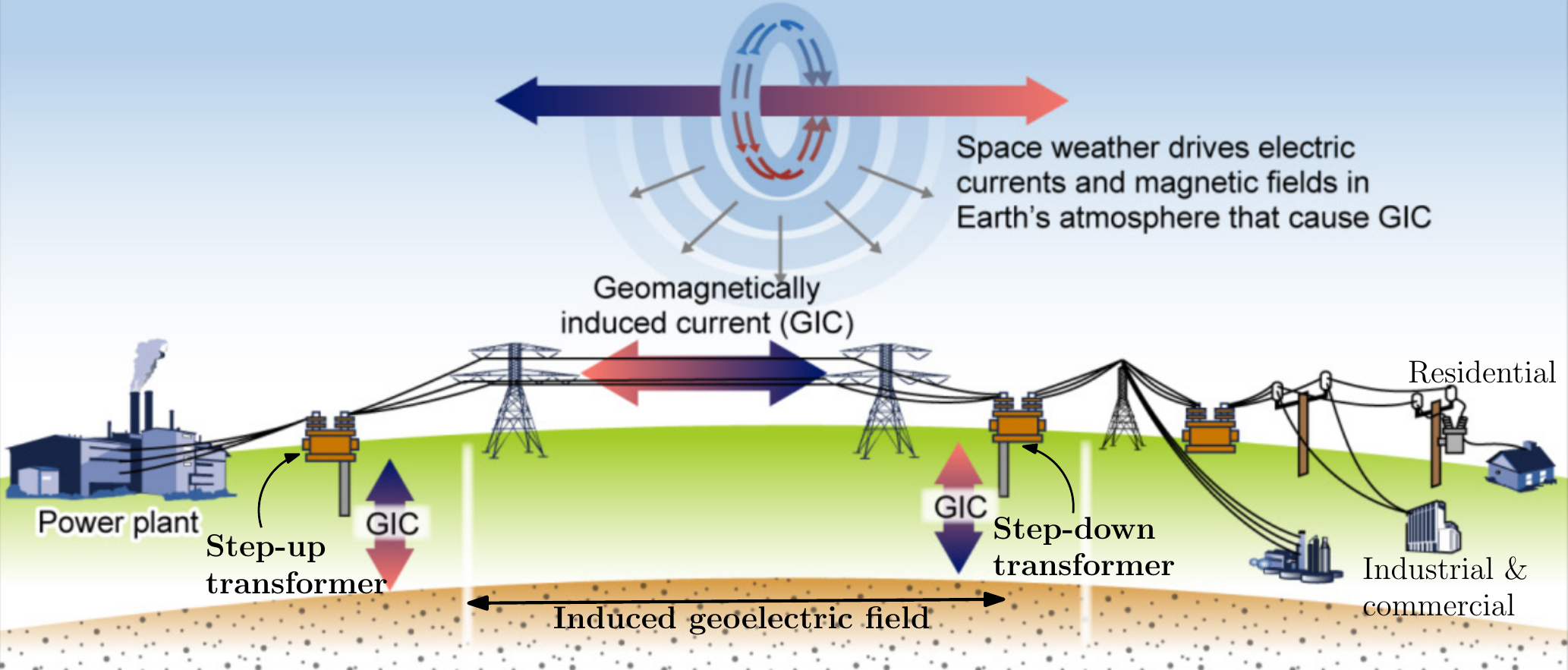}
    \caption{Effect of GMDs on the electric grid. (Source: \cite{gao_gmd_2018})}
    \label{fig:GIC_illustration}
\end{figure}

% [2] Literature review
To mitigate the harmful effects of GMDs, one can install Direct Current (DC) blocking devices at regional substations to prevent the GICs, which is quasi-DC, from entering the power network through transformer neutrals \cite{bolduc2005development}.
This has led to work like \cite{zhu2014blocking} which posed the GIC blocking device placement problem that minimizes the costs of selecting appropriate locations to install these devices.
Recently, numerous researchers have suggested that the risk of GICs could be reduced by the use of existing controls such as generator re-dispatch, line switching and load shedding. %\cite{lu2017optimal,kazerooni2018transformer}.
In \cite{lu2017optimal}, for the first time, the authors proposed an optimal transmission line switching (OTS) model under GMDs based on Alternating Current (AC) power flow equations and a set of constraints that captures GIC effects on various types of transformers. Utilizing state-of-the-art convex relaxations, the model was formulated as a mixed-integer quadratic convex program, which could be solved using commercial optimization solvers, albeit on small-scale instances. Recently, \cite{kazerooni2018transformer} presented heuristic-based algorithms to mitigate the effect of GICs on transformers by using line switching strategies on large-scale grids. 
Given that GMD events are hard to predict in advance and that the probability distribution of the uncertain magnitude and orientation of GMDs is not known precisely due to the insufficient historical data, \cite{lu2019distributionally} proposed a two-stage distributionally robust optimization (DRO) model with a mean-support ambiguity set and applied the standard column-and-constraint generation (CCG) algorithm \cite{zeng2013solving} to solve on a small-scale instance, with prohibitively slow run times. 

% [3] Contribution
In this paper, we formulate a modified and an improved version of the two-stage DRO formulation presented in \cite{lu2019distributionally}. 
While \cite{lu2019distributionally} is focused on making the mitigation actions (line switching and generator dispatch) that hedges against the worst-case expected load shedding costs calculated based on the power flow equations at the second-stage, this paper focuses on making not only the mitigation actions but also their corresponding power flow operations at the first stage so as to minimize the worst-case expected damage costs due to GICs. This proposed decomposition also better represents the operation of electric grids from a practical perspective. Specifically, the first-stage problem models the AC Optimal Transmission Switching (AC-OTS) which determines active transmission lines and the set-points for generators which minimize the worst-case expected costs occurred by GMDs, i.e., taking expectation over the worst-case probability distribution among all the distributions in the ambiguity set. Given the solutions of the first-stage, the second-stage problem consists of linear constraints that capture the GIC effects. We assume the mean-support ambiguity set is provided and is uniform throughout the grid. The support set of the uncertain parameters is convex and can be approximated by a polytope with $N$ extreme points. With these assumptions, the contributions of this paper are: 
(a) We first reformulate the two-stage DRO model as a min-max-min problem that can be solved by a CCG algorithm. \cite{lu2019distributionally} solves sub-problems which contain bilinear terms in the objective function, leading to weaker relaxations. To circumvent this issue, we solve a set of linear programs, each of which corresponds to an extreme point of the support set and enhances computation tractability. (b)  For the special case of the support set with \textit{three} extreme points ($N=3$), we prove that the two-stage DRO model can be equivalently reformulated as a two-stage stochastic program with three scenarios. We further propose to construct a monolithic reformulation which can be solved efficiently using commercial solvers, and (c) We present a detailed numerical analysis on the `epri-21' system, which is designed specifically for the GMD studies. 

\section{Mathematical Formulation}
\label{sec:math_form}
This section describes mathematical models that find a set of actions which mitigates the negative impacts of uncertain GMDs. Note that the formulation presented in this paper focuses on the quasi-static case (single time period) and leaves time-extended modeling for the future work. 
%--------------------------
%   Nomenclature            
%--------------------------
\begin{table}[h!]
    \footnotesize
    \centering
    \caption{Nomenclature}
    \label{tab-definition}    
    \begin{tabular}{l||l}
		\hline
		\multicolumn{2}{c}{\textbf{Sets and parameters}} \\ \hline
		$\mathcal{G},\ \mathcal{N}, \ \mathcal{E}$ & a set of generators, buses, and lines in AC network \\
		$\mathcal{E}^{\tau} \subseteq \mathcal{E} $ & a set of transformers \\
		$\mathcal{E}_i \subseteq \mathcal{E}$ & a set of lines connected to $i \in \mathcal{N}$ \\
		$c^{\text{\tiny F0}}_k$ & fixed cost when turning on $k \in \mathcal{G}$\\
		$c^{\text{\tiny F1}}_k, c^{\text{\tiny F2}}_k$ & fuel cost coefficients of power generation of $k \in \mathcal{G}$ \\
		$\LBgp_k, \UBgp_k$ & bounds on the real power generation of $k \in \mathcal{G}$ \\
		$\LBgq_k, \UBgq_k$ & bounds on the reactive power generation of $k \in \mathcal{G}$ \\
		$\kappa^{\text{\tiny l}}$ & unit penalty cost for power unbalance at $i \in \mathcal{N}$ \\
		$\ddp_i, \ddq_i$ & real and reactive power demand at $i \in \mathcal{N}$ \\
		$\underline{v}_i, \overline{v}_i$ & voltage limits at $i \in \mathcal{N}$ \\
		$g^{\text{\tiny s}}_i, b^{\text{\tiny s}}_i$ & shunt conductance and susceptance at $i \in \mathcal{N}$ \\
		$g_e, b_e$ & conductance, susceptance of $e \in \mathcal{E}$ \\
        $b_e^{\text{\tiny c}}$ & line charging susceptance of  $e \in \mathcal{E}$ \\		
		$\overline{s}_e$ & apparent power limit of line $e \in \mathcal{E}$\\
		$\alpha_{ij}$ & tap ratio of $e_{ij} \in \mathcal{E}$\\
		$k_e$ & loss factor of transformer $e \in \mathcal{E}^{\tau}$\\
		$\overline{I}^{\text{\tiny eff}}_e$ & upper limit of the effective GICs on $e \in \mathcal{E}^{\tau}$\\
		\hline 
		$\mathcal{N}^{\text{\tiny d}}, \mathcal{E}^{\text{\tiny d}}$ & a set of nodes and arcs in DC network \\
		$\mathcal{E}^{\text{\tiny d}-}_m, \mathcal{E}^{\text{\tiny d}+}_m$ & a set of incoming and outgoing arcs connected to $m \in \mathcal{N}^{\text{\tiny d}}$ \\
		$\gamma_{\ell}$ & conductance of $\ell \in \mathcal{E}^{\text{\tiny d}}$ \\
		$a_m$ & inverse of ground resistance at $m \in \mathcal{N}^{\text{\tiny d}}$  \\
		$\overline{v}^{\text{\tiny d}}$ & bound on the GIC-induced voltage magnitude \\
		$\tilde{\xi}_{\ell} $ & (random) GIC-induced voltage sources on $\ell \in \mathcal{E}^{\text{\tiny d}}$ \\ 
		\hline
% 		\hline
		\multicolumn{2}{c}{\textbf{Variables}} \\ \hline
		$\zza_e \in \mathbb{B}$ & $\zza_e=1$ if $e \in \mathcal{E}$ is turned on, and $\zza_e=0$ otherwise  \\
		$\zzg_k \in \mathbb{B}$ & $\zzg_k=1$ if $k \in \mathcal{G}$ is turned on, and $\zzg_k=0$ otherwise  \\		
		$ \lpp_i, \lpm_i$ & real power shedding at $i \in \mathcal{N}$ \\
		$ \lqp_i, \lqm_i$ & reactive power shedding at $i \in \mathcal{N}$ \\
		$v_i$ & voltage magnitude at $i \in \mathcal{N}$ \\
		$\theta_i$ & phase angle at $i \in \mathcal{N}$ \\
		$\ffp_k, \ \ffq_k$ & real and reactive power generated by $k \in \mathcal{G}$\\
		$p_{ei}, \ p_{ej}$ & real power flow on $e_{ij} \in \mathcal{E}$ \textit{from node} $i$ and \textit{to node} $j$ \\
		$q_{ei}, \ q_{ej}$ & reactive power flow on $e_{ij} \in \mathcal{E}$ \textit{from node} $i$ and \textit{to node} $j$ \\		
		$w_i$ & $w_i = v_i^2, \ \forall i \in \mathcal{N}$ \\
		$w^{\text{\tiny c}}_e$ & $w^{\text{\tiny c}}_e=v_i v_j \cos(\theta_i-\theta_j), \ \forall e_{ij} \in \mathcal{E}$ \\
		$w^{\text{\tiny s}}_e$ & $w^{\text{\tiny s}}_e=v_i v_j \sin(\theta_i-\theta_j), \ \forall e_{ij} \in \mathcal{E}$ \\
		$d^{\text{\tiny qloss}}_i$ & allowable reactive power loss due to GICs at $i \in \mathcal{N}$ \\		
		\hline 
		$I^{\text{\tiny d}}_{\ell}$ & GICs that flow on $\ell \in \mathcal{E}^{\text{\tiny d}}$ \\
        $I^{\text{\tiny eff}}_e$ & effective GICs on $e \in \mathcal{E}^{\tau}$ \\				
        $v^{\text{\tiny d}}_m$ & GIC-induced voltage magnitude at $m \in \mathcal{N}^{\text{\tiny d}}$  \\		
		\hline
	\end{tabular}
\end{table}
\subsection{AC and DC power network representation}
\label{subsec:ACDC}
\begin{figure}
    \centering
    \includegraphics[scale=0.4]{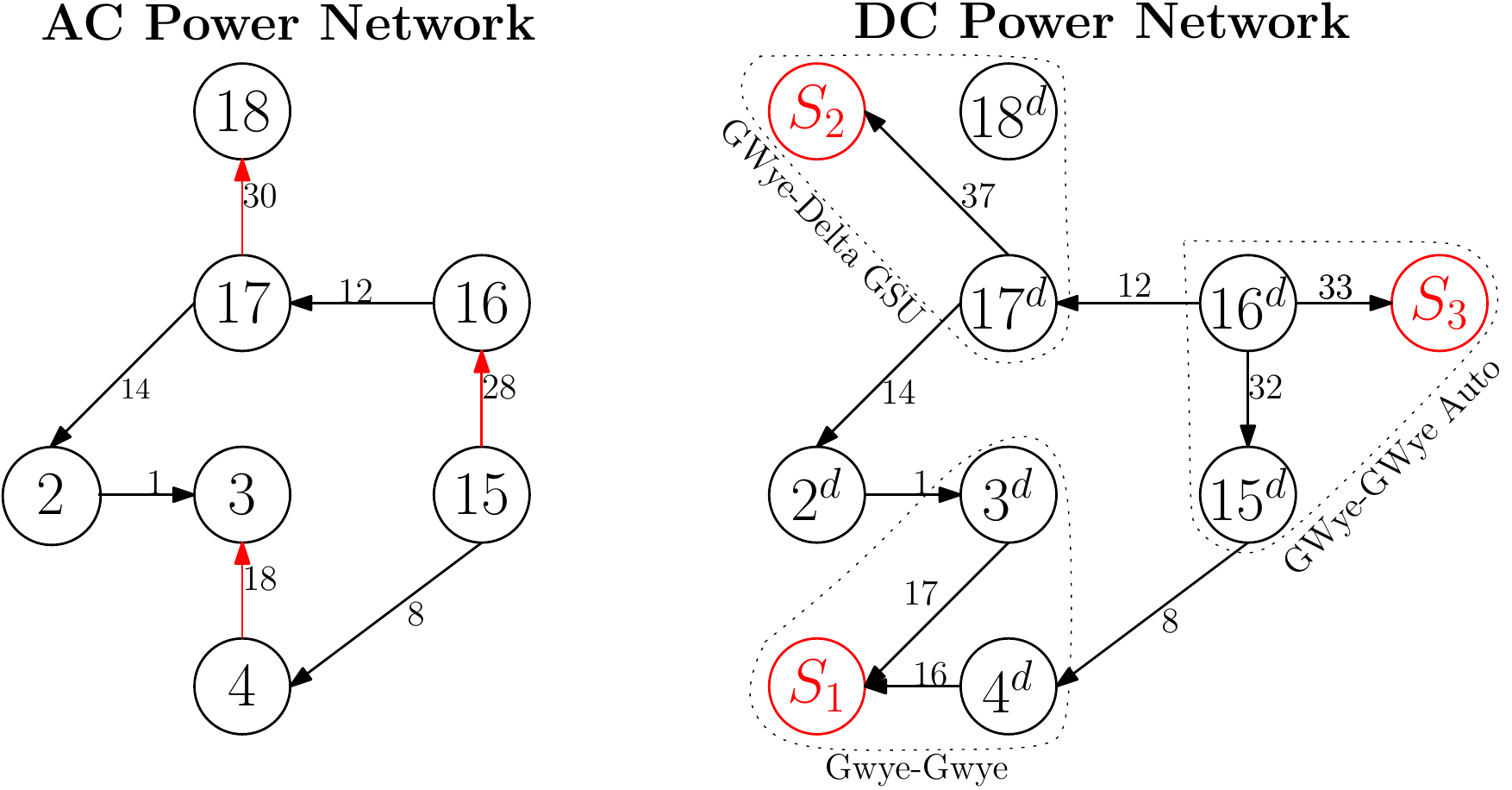}
    \caption{An example of AC to DC network mapping}
    \label{fig:ACDC}
    \vspace{-0.5cm}
\end{figure}

The AC power network is represented by a graph $(\mathcal{N}, \mathcal{E})$, where $\mathcal{N}$ is a set of nodes and $\mathcal{E}$ is a set of arcs. 
This network is the standard representation used for modeling AC power flow physics in power system applications.
%that is operated under the AC power flow physics. 
The set $\mathcal{E}$ is composed of $\mathcal{E}^{\tau}$, a set of transformers, and $\mathcal{E} \setminus \mathcal{E}^{\tau}$, a set of transmission lines. The set $\mathcal{N}$ is composed of buses that are adjacent to transmission lines and/or transformers.
For calculating GICs (details in subsequent sections), which is a quasi-DC flow, we construct a DC power network $(\mathcal{N}^{\text{\tiny d}}, \mathcal{E}^{\text{\tiny d}})$. $\mathcal{N}^{\text{\tiny d}}$ includes buses in $\mathcal{N}$ and additional nodes that model the neutral grounding points of transformers.
%indicate sub-stations, and 
The set $\mathcal{E}^{\text{\tiny d}}$ includes transmission lines in $\mathcal{E}$ and additional lines between the end points of transformers and their neutrals. The transformer configurations depend on the type of the transformer. In this paper, we consider 3 types of transformers, (i) Gwye-Gwye, (ii) GWye-GWye Auto, and (iii) GWye-Delta GSU. Detailed descriptions of these transformer types are found in \cite{lu2017optimal,lu2019distributionally}. 

The mapping between the AC and DC network representations are described in Figure \ref{fig:ACDC}.
In the AC network, $\mathcal{N} = \{2,3,4,15,16,17,18\}$, $\mathcal{E}=\{1,8,12,14,18,28,30\}$, and $\mathcal{E}^{\tau}=\{18,28,30\}$ (red arrows in Figure \ref{fig:ACDC}). The DC network $(\mathcal{N}^{\text{\tiny d}}, \mathcal{E}^{\text{\tiny d}})$ is constructed by adding information on different types of transformers and their connection to neutral (red circles in Figure \ref{fig:ACDC}). 
We first add neutral nodes $\{S_1,S_2,S_3\}$, each of which is connected to a transformer. The set $\mathcal{N}$ is relabeled with $\{2^d,3^d,4^d,15^d,16^d,17^d,18^d\}$ in the DC network. Topological mappings for different types of transformers are highlighted in Figure \ref{fig:ACDC}. 

To link the two networks, we define functions $E$ and $E^{-1}$ that map $\ell \in \mathcal{E}^{\text{\tiny d}}$ to an edge $e \in \mathcal{E}$ and vice versa, i.e., if $E_{\ell}=e$, then $E^{-1}_e = \{ \ell \in \mathcal{E}^{\text{\tiny d}} \ : \ E_{\ell} = e  \}$. 
For example, line $17$ in the DC network maps to the transformer line $18$ in the AC network, thus $E_{17}=18$. Transformer line $18$ maps to lines $16$ and $17$ in DC network, thus $E^{-1}_{18} = \{16, 17\}$. 
% \vspace{-0.2cm}
\subsection{GIC calculation}
This section describes the calculation of GICs and how they affect different types of transformers. Under the assumption of an uniformly induced geo-electric field within an interconnected electric grid, GIC that flows on a line in the DC network is given by 
\begin{align*}
I^{\text{\tiny d}}_{\ell} = \gamma_{\ell} (v^{\text{\tiny d}}_m - v^{\text{\tiny d}}_n + \tilde{\xi}_{\ell}), \ \forall \ell_{mn} \in \mathcal{E}^{\text{\tiny d}}
\end{align*}
where, $\tilde{\xi}_{\ell}$ is the GIC-induced voltage source, given by: 
\begin{align*}
\tilde{\xi}_{\ell} = 
\begin{cases}
\tilxiE L^{\text{\tiny E}}_{\ell} + \tilxiN L^{\text{\tiny N}}_{\ell}, \ \forall \ell \in \mathcal{E}^{\text{\tiny d}} : E_{\ell} \in \mathcal{E} \setminus \mathcal{E}^{\tau}, \\ 
0, \ \forall \ell \in \mathcal{E}^{\text{\tiny d}} : E_{\ell} \in \mathcal{E}^{\tau}
\end{cases}
\end{align*}
and $\tilxiE$ and $\tilxiN$ are \textit{uncertain} geo-electric fields [V/km] in the eastward and northward direction, respectively, and $L^{\text{\tiny E}}_{\ell}$ and $L^{\text{\tiny N}}_{\ell}$ are the lengths [km] of transmission lines $\ell$ in the eastward and northward direction, respectively.

Based on the calculated GIC in the DC network, the effective GIC of a transformer in the AC network is given by
\begin{align*}
& I^{\text{\tiny eff}}_e = | \Theta(I^{\text{\tiny d}}_{\ell}, \ \forall \ell \in E^{-1}_e )|, \ \forall e \in \mathcal{E}^{\tau}
\end{align*}
where $\Theta(I^{\text{\tiny d}}_{\ell}, \ \forall \ell \in E^{-1}_e )$ is a linear function of GIC ($I^{\text{\tiny d}}_{\ell}$). This function depends on the type of transformer as described in Table \ref{tab-effGIC}.
Note that $(N_h, N_l,N_s, N_c)$ are parameters which indicate the number of turns in the high-side/low-side/series/common winding, respectively.
\begin{table}[h!]
    \small
    \centering
    \caption{Effective GICs for each type of transformers}
    \label{tab-effGIC}    
    \begin{tabular}{l|l|l}
    \toprule	
    Type of transformer $e$ & $E^{-1}_e$ 	& $\Theta(I^{\text{\tiny d}}_{\ell}, \ \forall \ell \in E^{-1}_e )$ \\ 
    \hline
    Gwye-Gwye		& $\{h, l\}$  	& $\Theta(I^{\text{\tiny d}}_{h}, I^{\text{\tiny d}}_{l}) = \frac{ N_h I^{\text{\tiny d}}_{h} + N_l I^{\text{\tiny d}}_{l}}{N_h}$ \\
    GWye-GWye Auto	& $\{s, c\}$  	& $\Theta(I^{\text{\tiny d}}_{s}, I^{\text{\tiny d}}_{c}) = \frac{N_s I^{\text{\tiny d}}_{s} + N_c I^{\text{\tiny d}}_{c}}{N_s+N_c}$	 \\
    GWye-Delta GSU	& $\{h\}$ &  $\Theta(I^{\text{\tiny d}}_{h}) = I^{\text{\tiny d}}_{h} $ \\ 
    \bottomrule
    \end{tabular}
\end{table}

Lastly, given $\mathcal{E}^{\tau}_i$ is a set of transformers connected to node $i$, the reactive power loss \cite{zhu2014blocking} due to GICs at node $i$ in the AC network is calculated by
\begin{align}
\sum_{e \in \mathcal{E}^{\tau}_i} k_e v_i I^{\text{\tiny eff}}_e, \ \forall i \in \mathcal{N}
\end{align}

\subsection{Uncertainty set description}
\label{subsec:uncertain}
\vspace{-0.1cm}
\begin{figure}[h!]
	\centering
	\includegraphics[scale=0.35]{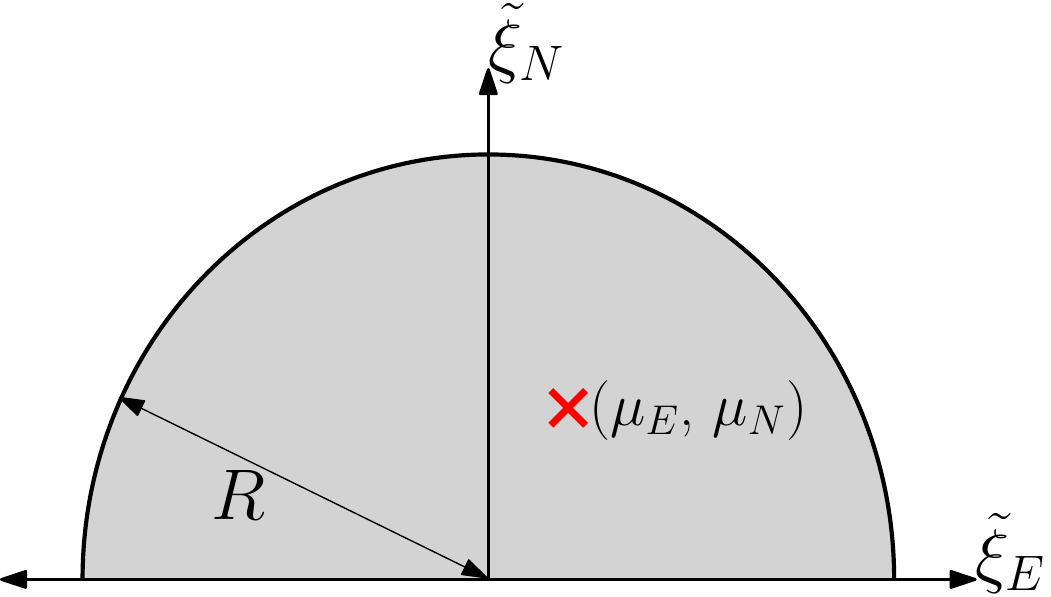}
	\caption{Support set and mean of uncertain GMDs $(\tilxiE, \tilxiN)$ }
	\label{fig:SupportSets}
\end{figure}
In this model, the geo-electric field $(\tilxiE, \tilxiN)$ is uncertain when planning decisions are made
to mitigate the risk of GMDs. We assume that the support set and mean values of $(\tilxiE, \tilxiN)$ are provided or are extracted from a set of historical data, but their joint probability distribution is unknown. Generally, there is not enough information to construct reasonable probability distributions \cite{woodroffe2016latitudinal}.
In this setting, we construct the following mean-support ambiguity set $\mathcal{D}$:
\begin{align*}
\mathcal{D} := \{ \mathbb{P} \ : \  \mathbb{E}_{\mathbb{P}_{\text{\tiny E}}} [ \tilxiE ] = \mu_{\text{\tiny E}}, \ \  \mathbb{E}_{\mathbb{P}_{\text{\tiny N}}} [ \tilxiN ] = \mu_{\text{\tiny N}}, \ \ \mathbb{P}\{ \tilxi \in \Xi \}=1  \} 
\end{align*}
where $(\mathbb{P}_{\text{\tiny E}}, \mathbb{P}_{\text{\tiny N}})$ and $(\mu_{\text{\tiny E}}, \mu_{\text{\tiny N}})$ are the marginal distributions and mean values of $(\tilxiE, \tilxiN)$, respectively. The support set $\Xi$ of $\tilxi=(\tilxiE, \tilxiN)$ is defined as
% \vspace{-0.2cm}
\begin{align*}
\Xi := \left\{ \tilxi=(\tilxiE, \tilxiN) \in \mathbb{R}^2 : -R \leq \tilxiE \leq R, \ \ 0 \leq \tilxiN \leq R, \right. \\ \left. (\tilxiE)^2 + (\tilxiN)^2 \leq R^2 \right\} 
\end{align*}
where $R$ is a radius of half-circle support set as shown in Figure \ref{fig:SupportSets}. 
%{\color{blue} Add some text to substantiate the half-circle support set}
Practically speaking, the support set is bounded by an estimate of the worst-case magnitude of the storm. Since the square of the magnitude is equal to the sum of the squares of the northward and eastward field strength, a magnitude bound yields a half circle support set for $(\tilxiE, \tilxiN)$ \cite{woodroffe2016latitudinal}.

\subsection{Two-stage DRO formulation}
\label{subsec:2stage}
In this section we describe a two-stage DRO formulation which determines a set of mitigation actions which minimizes the cost of generation and a penalty for shedding loads and is robust to the worst-case probability distribution in the ambiguity set $\cal{D}$ as defined in section \ref{subsec:uncertain}. To preserve computational tractability, we use the second-order conic relaxations of the rectangular form of the nonlinear, nonconvex AC power flow constraints and formulate the first-stage problem as an MISOCP: 

\begin{subequations} 
	\small{
	\begin{align}
	\min \ &  \sum_{k \in \mathcal{G}} ( c^{\text{\tiny F0}}_k \zzg_k + c^{\text{\tiny F1}}_k \ffp_k +  c^{\text{\tiny F2}}_k (\ffp_k)^2 ) + \sum_{i \in \mathcal{N}} \kappa^{\text{\tiny l}} ( \lpp_i + \lpm_i + \lqp_i + \lqm_i  ) \nonumber \\ 
	 &\hspace{102pt} + \sup_{\Pb \in \mathcal{D}} \Eb_{\Pb}[ \mathcal{Q}(\boldsymbol{{\zza}, v, {d^{\text{\tiny qloss}}}, {\tilde{\xi}}}] \label{eq:stg1_obj}  \\ 
	\mbox{s.t.} \ 
	% Balance 					
	& \sum_{e \in \mathcal{E}_i } p_{ei} = \sum_{k \in \mathcal{G}_i } \ffp_k - \ddp_i + \lpp_i - \lpm_i - g^{\text{\tiny s}}_i w_i, \ \forall i \in \mathcal{N}, \label{eq:stg1_1} \\
	& \sum_{e \in \mathcal{E}_i } q_{ei} = \sum_{k \in \mathcal{G}_i } \ffq_k - \ddq_i  + \lqp_i - \lqm_i + b^{\text{\tiny s}}_i w_i - d^{\text{\tiny qloss}}_i, \ \forall i \in \mathcal{N}, \label{eq:stg1_2} \\
	% Thermal Limit 
	& p_{ei}^2 + q_{ei}^2 \leq \zza_e (\overline{s}_e^2) , \ \ p_{ej}^2 + q_{ej}^2 \leq \zza_e (\overline{s}_e^2), \ \forall e_{ij} \in \mathcal{E}, \label{eq:stg1_3} \\	
	% SOCP Power flow
	& p_{ei} =   \frac{1}{\alpha_{ij}^2} g_e w^{\text{\tiny z}}_{ei} - \frac{1}{\alpha_{ij}} (g_e w^{\text{\tiny cz}}_e + b_e w^{\text{\tiny sz}}_e), \ \forall e_{ij} \in \mathcal{E}, \label{eq:stg1_4}  \\
	& p_{ej} =   g_e w^{\text{\tiny z}}_{ej} - \frac{1}{\alpha_{ij}}  (g_e w^{\text{\tiny cz}}_e - b_e w^{\text{\tiny sz}}_e  ), \ \forall e_{ij} \in \mathcal{E}, \label{eq:stg1_5} \\
	& q_{ei} =   -\frac{1}{\alpha_{ij}^2} (b_e + \frac{b_e^{\text{\tiny c}}}{2}) w^{\text{\tiny z}}_{ei} + \frac{1}{\alpha_{ij}} (b_e w^{\text{\tiny cz}}_e - g_e w^{\text{\tiny sz}}_e  ), \ \forall e_{ij} \in \mathcal{E},  \label{eq:stg1_6} \\
	& q_{ej} =   -(b_e + \frac{b_e^{\text{\tiny c}}}{2}) w^{\text{\tiny z}}_{ej} + \frac{1}{\alpha_{ij}} (b_e w^{\text{\tiny cz}}_e + g_e w^{\text{\tiny sz}}_e  ),  \forall e_{ij} \in \mathcal{E}, \label{eq:stg1_7} \\
	& w^{\text{\tiny z}}_{ei} \in \langle \zza_e, w_i\rangle^{\text{\tiny MC}}, \ \ w^{\text{\tiny z}}_{ej} \in \langle \zza_e, w_j\rangle^{\text{\tiny MC}}, \ \forall e_{ij} \in \mathcal{E}, \label{eq:stg1_8}\\ 		
	& \underline{w}^{\text{\tiny c}}_e \zza_e \leq w^{\text{\tiny cz}}_{e} \leq \overline{w}^{\text{\tiny c}}_e \zza_e, \ \ \underline{w}^{\text{\tiny s}}_e \zza_e \leq w^{\text{\tiny sz}}_{e} \leq \overline{w}^{\text{\tiny s}}_e \zza_e,  \ \forall e \in \mathcal{E}, \label{eq:stg1_9}\\			
	& (w^{\text{\tiny cz}}_{e})^2 + (w^{\text{\tiny sz}}_{e})^2 \leq w^{\text{\tiny z}}_{ei} w^{\text{\tiny z}}_{ej}, \ \forall e_{ij} \in \mathcal{E}, \label{eq:stg1_10} \\	
	& \tan (\underline{\theta}_{ij}) w^{\text{\tiny cz}}_{e} \leq w^{\text{\tiny sz}}_{e} \leq \tan (\overline{\theta}_{ij}) w^{\text{\tiny cz}}_{e}, \ \forall e_{ij} \in \mathcal{E}, \label{eq:stg1_11}\\	
	& (v_i)^2 \leq w_i \leq (\overline{v}_i + \underline{v}_i) v_i - \overline{v}_i \underline{v}_i, \ \forall i \in \mathcal{N}, \label{eq:stg1_12}\\		
	%% Relaxed Constraints 
	& \sum_{e \in \mathcal{E}_k} \zza_e \geq \zzg_k, \  \forall k \in \mathcal{G}, \label{eq:stg1_13} \\	
	& \underline{v}_i \leq v_i \leq \overline{v}_i, \ \ \lpp_i, \lpm_i, \lqp_i, \lqm_i \geq 0, \ \forall i \in \mathcal{N},  \label{eq:stg1_14} \\
	& \LBgp_k \zzg_k \leq \ffp_k \leq \UBgp_k \zzg_k, \ \ \LBgq_k \zzg_k \leq \ffq_k \leq \UBgq_k \zzg_k, \ \forall k \in \mathcal{G}, \label{eq:stg1_15} \\
	& \zza_e \in \{0,1\}, \ \forall e \in \mathcal{E}, \ \ \zzg_k \in \{0,1\}, \ \forall k \in \mathcal{G}. \label{eq:stg1_16}
	\end{align}}
	\label{eq:stg1}
\end{subequations}
where, given the values of $\boldsymbol{{\zza}, v, {d^{\text{\tiny qloss}}}}$ from the first-stage problem \eqref{eq:stg1} and the realization of $\boldsymbol{\tilde{\xi}}$,  $\mathcal{Q}(\boldsymbol{{\zza}, v, {d^{\text{\tiny qloss}}}, {\tilde{\xi}}})$ is the optimal value of the following second-stage problem which evaluates GICs in the DC network:

\begin{subequations}
\small{
\begin{align}
	\min \ & \sum_{i \in \mathcal{N}} \kappa^{\text{\tiny s}} s_i  \\
	\mbox{s.t.} \ 
	%% GIC calculation
	& \frac{I^{\text{\tiny d}}_{\ell}}{\gamma_{\ell}} \leq (v^{\text{\tiny d}}_m - v^{\text{\tiny d}}_n + \tilde{\xi}_{\ell}) + M^{-}_{\ell} (1-\zza_{E_{\ell}}), \ \forall \ell_{mn} \in \mathcal{E}^{\text{\tiny d}}, \label{eq:stg2_1} \\
	& \frac{I^{\text{\tiny d}}_{\ell}}{\gamma_{\ell}} \geq (v^{\text{\tiny d}}_m - v^{\text{\tiny d}}_n + \tilde{\xi}_{\ell}) - M^{+}_{\ell} (1-\zza_{E_{\ell}}), \ \forall \ell_{mn} \in \mathcal{E}^{\text{\tiny d}}, \label{eq:stg2_2} \\
	& - \zza_{E_{\ell}} M^{-}_{\ell} \leq \frac{I^{\text{\tiny d}}_{\ell}}{\gamma_{\ell}} \leq  \zza_{E_{\ell}} M^{+}_{\ell}, \ \forall \ell_{mn} \in \mathcal{E}^{\text{\tiny d}}, \label{eq:stg2_3}\\
	& \sum_{\ell \in \mathcal{E}^{d-}_m} I^{\text{\tiny d}}_{\ell} - \sum_{\ell \in \mathcal{E}^{d+}_m} I^{\text{\tiny d}}_{\ell} = a_m v^{\text{\tiny d}}_m, \ \forall m \in \mathcal{N}^{\text{\tiny d}}, \label{eq:stg2_4}\\
	& I^{\text{\tiny eff}}_{e} \geq \Theta(I^{\text{\tiny d}}_{\ell}, \ \forall \ell \in E^{-1}_e ), \ \ I^{\text{\tiny eff}}_{e} \geq -\Theta(I^{\text{\tiny d}}_{\ell}, \ \forall \ell \in E^{-1}_e ), \label{eq:stg2_5} \\ 
	& 0 \leq I^{\text{\tiny eff}}_e \leq \overline{I}^{\text{\tiny eff}}_e, \ \forall e \in \mathcal{E}^{\tau}, \quad s_i \geq 0, \ \forall i \in \mathcal{N} \label{eq:stg2_6}\\	
	& u^{\text{\tiny d}}_{ei} \in \langle v_i, I^{\text{\tiny eff}}_{e} \rangle^{\text{\tiny MC}}, \ d^{\text{\tiny qloss}}_i \geq \sum_{e \in \mathcal{E}^{\tau}_i} k_e u^{\text{\tiny d}}_{ei} - s_i,  \ \forall i \in \mathcal{N}. \label{eq:stg2_7}
\end{align}}    
\end{subequations}
Constraints \eqref{eq:stg1_1}--\eqref{eq:stg1_15} model the relaxed AC power flow equations. These constraints include mitigation actions such as transmission line switching and switching of generators. Note that generator switching in this formulation \textit{does not} imply that we model economic unit commitment. Constraints \eqref{eq:stg1_1} and \eqref{eq:stg1_2} model real and reactive power balance constraints, including 
allowable GIC-induced reactive power loss ($d_i^{\text{\tiny qloss}}$).
%the $d_i^{\text{\tiny qloss}}$ term which accounts for the allowable reactive power loss due to GICs. 
Constraints \eqref{eq:stg1_3} ensure that the apparent power flow does not exceed its limit when the line is closed. Constraints \eqref{eq:stg1_4}--\eqref{eq:stg1_7} model Ohm's law. Constraints \eqref{eq:stg1_8}--\eqref{eq:stg1_12} model the MISOCP relaxation of the AC power flow equations using the formulations discussed in \cite{kocuk2017new}. Notation $\langle x_i,x_j\rangle^{\text{\tiny MC}}$ models the standard McCormick relaxation of a bilinear term $x_i\cdot x_j$, which is very effective for ACOPF problems \cite{kocuk2017new,lu2018tight,narimani2018comparison,nagarajan2016optimal,molzahn2019survey}, and also for generic nonlinear programs with bilinear terms \cite{nagarajan2016tightening}. Constraints \eqref{eq:stg1_13} ensure that a generator is turned off when all the lines and transformers connected to that generator are off. 

In the second-stage problem, constraints \eqref{eq:stg2_1}--\eqref{eq:stg2_3} calculate valid GICs in the DC network for lines that are switched on. Else, they are deactivated with the big-M coefficients, where $M^{+}_{\ell} = \overline{v}^{\text{\tiny d}} + \tilde{\xi}_{\ell}$, $M^{-}_{\ell} = \overline{v}^{\text{\tiny d}} - \tilde{\xi}_{\ell}$. Constraints \eqref{eq:stg2_4} model nodal balance equation for GICs in the DC network. Note that, in \eqref{eq:stg2_4}, $a_m=0$ when $m$ is not a grounded neutral node. Constraints \eqref{eq:stg2_5} and \eqref{eq:stg2_6} calculate the effective GICs for each type of transformer (see Table \ref{tab-effGIC}), which in-turn is used to calculate the reactive power losses induced in the AC network, as shown in constraints \eqref{eq:stg2_7}. 

\noindent
\textbf{Interpretation of the two-stage DRO formulation} Once the uncertain GMDs ($\tilxi$) are realized and the decisions in the AC network on line switching ($\boldsymbol{\zza}$), voltage magnitude ($\boldsymbol{v}$), and allowable reactive power losses $\boldsymbol{d^{\text{\tiny qloss}}}$ (interpreted as the amount of reactive power loss that will not cause excessive voltage drops. Exceeding this value would correspond to the cost of installing or using a device to counteract the reactive losses), the second-stage problem calculates the effective GICs and the actual reactive power losses, given by $\sum_{e \in \mathcal{E}^{\tau}_i} k_e u^{\text{\tiny d}}_{ei}$. If the calculated reactive power losses in second-stage exceeds the allowable $\boldsymbol{d^{\text{\tiny qloss}}}$ from the first-stage, then appropriate mitigation actions are updated in the AC network to mitigate the negative effects on the transformers. Implicitly, this formulation assumes that reactive losses smaller than $\boldsymbol{d^{\text{\tiny qloss}}}$ will not cause voltage problems (i.e., a high voltage).

In summary, the proposed two-stage DRO formulation minimizes the power generation cost, penalty cost for shedding loads and the worst-case expected cost occurred by damaged transformers. 

\section{Solution methodologies}
\label{subsec:algo}
In this section, we describe several solution approaches to the two-stage DRO formulation. For ease of exposition, we rewrite the formulation using matrix notation as follows: 
\begin{subequations}
	\label{DRO_ORG}
	\begin{align}
	\min_{\zb \in \mathcal{F}} \ & \qb^{\text{\tiny T}} \zb + \sup_{\mathbb{P} \in \mathcal{D}} \mathbb{E}_{\mathbb{P}}[ \mathcal{Q} (\zb, \tilxi) ] 
	\end{align}
	where $\mathcal{Q} (\zb, \tilxi)$ is the optimal value of the following problem:
	\begin{align}
	\mathcal{Q} (\zb, \tilxi) = \min_{\xb \in \mathcal{X}(\zb, \tilxi)} \ \cb^{\text{\tiny T}} \xb. 
	\end{align}
\end{subequations}	
Note that $\zb$ and $\mathcal{F}$ are the solution vector and the feasible region, respectively, of the first-stage problem \eqref{eq:stg1} and  $\mathcal{X}(\zb, \tilxi) = \{ \xb \in \mathbb{R}^{n} \ : \ \Ab \xb + \Bb(\tilxi) \zb \geq \db \}$ is the feasible region of the second-stage problem.
%-------------------------
% Proposition - II.1   
%-------------------------
\begin{proposition} 
\label{prop-1}
Problem \eqref{DRO_ORG} is equivalent to the following $\min$-$\max$-$\min$ problem:	
\begin{subequations}
\small
\begin{align*}
\min_{\zb \in \mathcal{F}} \ & \qb^{\text{\tiny T}} \zb + \mu_{\text{\tiny E}} \lambda_{\text{\tiny E}} + \mu_{\text{\tiny N}} \lambda_{\text{\tiny N}} + \bigg\{ \max_{\tilxi \in \Xi} \  \bigg( \min_{x \in \mathcal{X}(\zb, \tilxi)} \ \cb^{\text{\tiny T}} \xb \bigg) - \lambda_{\text{\tiny E}} \tilxiE - \lambda_{\text{\tiny N}}  \tilxiN \bigg\}.
\end{align*}    
\end{subequations}
\end{proposition}

\begin{proof}
The worst-case expected value $\sup_{\mathbb{P} \in \mathcal{D}} \mathbb{E}_{\mathbb{P}}[ \mathcal{Q} (\zb, \tilxi) ]$ can be written as
\begin{subequations}
\small
\begin{align*}
\max \ & \int_{\tilxi \in \Xi}  \mathcal{Q} (\zb, \tilxi) \ d\mathbb{P} \\
\mbox{s.t.} \
& \int_{\tilxi \in \Xi} \tilxiE \ d\mathbb{P} = \mu_{\text{\tiny E}}, \ \int_{\tilxi \in \Xi} \tilxiN \ d\mathbb{P} = \mu_{\text{\tiny N}}, \ \int_{\tilxi \in \Xi} \ d\mathbb{P} = 1. 
\end{align*}
\end{subequations}
By taking the dual, we obtain
\begin{align*}
\min \ & \mu_{\text{\tiny E}} \lambda_{\text{\tiny E}} + \mu_{\text{\tiny N}} \lambda_{\text{\tiny N}} + \eta \\
\mbox{ s.t. } \ & \lambda_{\text{\tiny E}} \tilxiE + \lambda_{\text{\tiny N}}  \tilxiN  + \eta \geq  \ \mathcal{Q} (\zb, \tilxi), \ \forall \tilxi \in \Xi. 
\end{align*}
where $\lambda_{\text{\tiny E}}$, $\lambda_{\text{\tiny N}}$, and $\eta$ are dual variables.
At optimality, we have $\eta^* = \max_{\tilxi \in \Xi} \ \{ \mathcal{Q} (\zb, \tilxi) - \lambda_{\text{\tiny E}} \tilxiE - \lambda_{\text{\tiny N}}  \tilxiN \}$, which leads to the proposed formulation.
\end{proof} 

%-------------------------
% Proposition - II.2   
%-------------------------
\begin{proposition}
\label{prop-2}
Given solution ($\zb$, $\lambda_{\text{\tiny E}}$, $\lambda_{\text{\tiny N}}$), the inner $\max$-$\min$ problem can be reformulated as the following $\max$ problem:
\begin{subequations}
\small
	\label{Subproblem}	
	\begin{align}
	\mathcal{Z}(\zb,\lambda_{\text{\tiny E}}, \lambda_{\text{\tiny N}}, \tilxi) = \max_{\tilxi \in \Xi} \ &  \cb^{\text{\tiny T}} \xb - \lambda_{\text{\tiny E}} \tilxiE - \lambda_{\text{\tiny N}}  \tilxiN  \\
	\mbox{s.t.} \ 
	& 0 \leq \Ab \xb + \Bb(\tilxi) \zb - \db \leq M(\bm{1}-\bm{\alpha}), \nonumber\\
	& 0 \leq \betab \leq M \bm{\alpha}, \  \Ab^{\text{\tiny T}} \betab = \cb, \ \bm{\alpha} \in \mathbb{B}^{m}. \nonumber
	\end{align}
\end{subequations}
where $M$ is a sufficiently large value.
\end{proposition}
\begin{proof}
Consider the following $\max$-$\min$ problem:
\begin{subequations}
\small
\begin{align*}
\max_{\tilxi \in \Xi} \ \bigg( \min_{x \in \mathcal{X}(\zb, \tilxi)} \ \cb^{\text{\tiny T}} \xb \bigg)  - \lambda_{\text{\tiny E}}  \tilxiE - \lambda_{\text{\tiny N}}  \tilxiN 
\end{align*}
\end{subequations}
Using the complementary slackness condition, we obtain
\begin{subequations}
\small
	\begin{align}
	\max_{\tilxi \in \Xi, \xb \in \mathbb{R}^n, \betab \in \mathbb{R}^m_+} \ &  \cb^{\text{\tiny T}} \xb - \lambda_{\text{\tiny E}} \tilxiE - \lambda_{\text{\tiny N}}  \tilxiN  \\
	\mbox{s.t.} \ 
	& \Ab \xb + \Bb(\tilxi) \zb \geq \db, \ \Ab^{\text{\tiny T}} \betab = \cb, \label{PD-feasiblity} \\
	& \betab^{\text{\tiny T}} (\Ab \xb + \Bb(\tilxi) \zb - \db ) = 0. \label{complementary}
	\end{align}
\end{subequations}
where $\betab$ is a dual vector, \eqref{PD-feasiblity} is primal and dual feasibility and \eqref{complementary} is the complementary slackness.
By introducing a big-$M$ coefficient, constraints \eqref{complementary} can be expressed as linear constraints, which leads to the formulation in \eqref{Subproblem}.
\end{proof}

\subsection{Column-and-Constraint Generation (CCG) Algorithm}
\label{subsec:ccg}
Using propositions \eqref{prop-1} and \eqref{prop-2}, the two-stage DRO problem with a convex support-set
%, as described in section \ref{subsec:uncertain}, 
is exactly solvable with the CCG algorithm \cite{zeng2013solving}. Since
the convex support set has infinitely many extreme points the CCG approach can be computationally intensive.
%However, the CCG approach can be computationally intensive due to the infinitely many extreme points of the convex support set and the MISOCP subproblem \eqref{Subproblem} with big-$M$ coefficients, which needs to be solved in the iterations of the CCG algorithm.
To improve the computational performance of the CCG, we develop \textit{two} tractable solution approaches that provide upper and lower bounds on the optimal value of problem \eqref{DRO_ORG}. 

\vspace{0.4cm}
\noindent
\underline{\textit{Polyhedral support set}}
% \vspace{-0.4cm}
\begin{figure}[h!]
	\centering
	\includegraphics[scale=0.28]{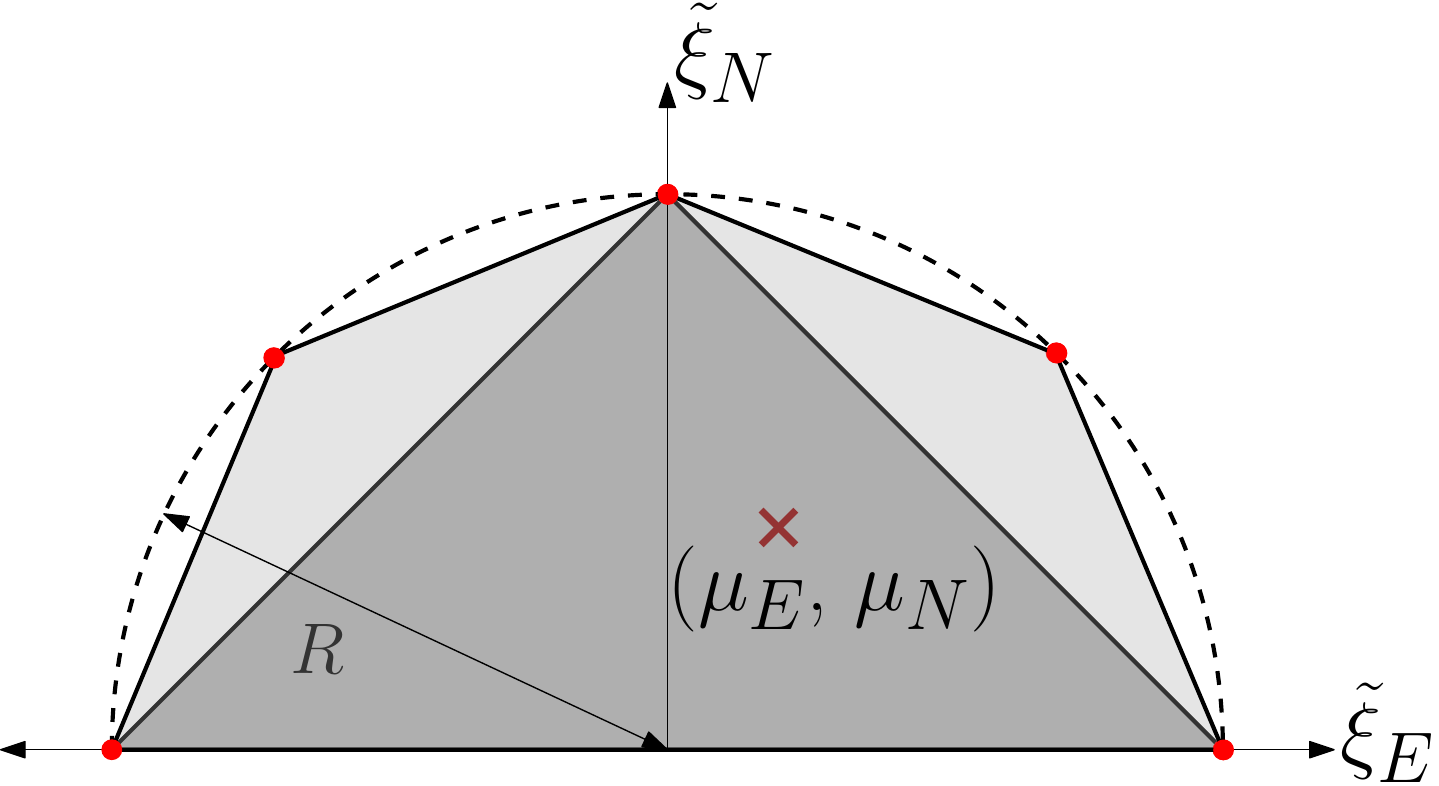}
	\caption{Polyhedral support sets with $3$ and $5$ extreme points.}
	\label{fig:poly}
\end{figure}
% \vspace{-0.2cm}
\begin{algorithm}[!h]
    \small
	\caption{\small CCG for $2$-stage DRO with polyhedral support set}
	\label{algo:CCG}
	\begin{algorithmic}[1]
		\STATE Set LB=$-\infty$, UB=$\infty$, $t=0$ and $\mathcal{T}=\emptyset$.
		\STATE Update LB by solving the following master problem:
		\begin{subequations}
			\begin{align}
			\text{LB} = \min_{\zb \in \mathcal{F}} \ & \qb^{\text{\tiny T}} \zb + \mu_{\text{\tiny E}} \lambda_{\text{\tiny E}} + \mu_{\text{\tiny N}} \lambda_{\text{\tiny N}} + \eta \\ 
			\mbox{ s.t. } \ 
			& \eta \geq \cb^{\text{\tiny T}} \xb^{\ell} - \lambda_{\text{\tiny E}} {\xi}_{\text{\tiny E}}^{\ell} - \lambda_{\text{\tiny N}} {\xi}_{\text{\tiny N}}^{\ell} , \ \forall \ell \in \mathcal{T}, \\
			& \Ab \xb^{\ell} + \Bb(\tilxi^{\ell}) \zb \geq \db, \ \forall \ell \in \mathcal{T}.
			\end{align}
		\end{subequations}
		-- Record an optimal solution $\zb^*, \lambda_{\text{\tiny E}}^*, \lambda_{\text{\tiny N}}^*$, and $\eta^*$.
		\STATE Solve the following problem:
		\begin{align*}
		\max_{\tilxi \in \Xi^N} \ \bigg\{ \mathcal{Q}(\zb^*, \tilxi) - \lambda_{\text{\tiny E}}^*  \tilxiE - \lambda_{\text{\tiny N}}^*  \tilxiN \bigg\}
		\end{align*}		
		 -- Record an optimal $\tilxi^*$ and the optimal value $\mathcal{Z}(\zb^*,\lambda_{\text{\tiny E}}^*, \lambda_{\text{\tiny N}}^*, \tilxi^*)$
		\STATE Update UB by 
		\begin{align*}
		\text{UB} = \min \{ \text{UB}, \ \qb^{\text{\tiny T}} \zb^* + \mu_{\text{\tiny E}} \lambda_{\text{\tiny E}}^* + \mu_{\text{\tiny N}} \lambda_{\text{\tiny N}}^* + \mathcal{Z}(\zb^*,\lambda_{\text{\tiny E}}^*, \lambda_{\text{\tiny N}}^*, \tilxi^*) \}
		\end{align*}
		\IF {$(\text{UB}-\text{LB})/\text{UB} \leq \epsilon$} 
		\STATE Stop and return $\zb^*$ as an optimal solution.
		\ELSE 
		\STATE Update ${\xi}_{\text{\tiny E}}^{t+1}=\tilde{\xi}_{\text{\tiny E}}^*$, ${\xi}_{\text{\tiny N}}^{t+1}=\tilde{\xi}_{\text{\tiny N}}^*$, $\mathcal{T} = \mathcal{T} \cup \{t+1\}$ and $t=t+1$.
		\STATE Go to step 2 and solve the updated master problem.
		\ENDIF			
	\end{algorithmic}
\end{algorithm}
\\
The first approach approximates the convex support set with a polytope that has $N$ extreme points.
When the polytope is a subset (resp. superset) of the support set (see Figure \ref{fig:poly}), it lower (resp. upper) bounds the optimal value of the original problem \eqref{DRO_ORG}.
In this section, we solve problem \eqref{DRO_ORG} with the polytope $\Xi^{\text{\tiny N}}$ defined as $\text{conv} \{ (\hat{\xi}_{\text{\tiny E}}^1, \hat{\xi}_{\text{\tiny N}}^1), \ldots, (\hat{\xi}_{\text{\tiny E}}^N, \hat{\xi}_{\text{\tiny N}}^N) \}$. Since $\Xi^N$ is also a convex set, strong duality (proposition \ref{prop-1}) holds. Thus, problem \eqref{DRO_ORG} with support set $\Xi^N$ is also solvable using the CCG algorithm (algorithm \eqref{algo:CCG}). Note that the CCG algorithm is guaranteed to converge in a finite number of iterations since the number of extreme points of $\Xi^N$ is finite. 

The CCG algorithm is a cutting plane method. The algorithm iteratively refines the feasible domain of the two-stage DRO problem by sequentially generating a set of recourse variables and their associated constraints. 
In algorithm \eqref{algo:CCG}, LB and UB denote the incumbent lower and upper bounds of problem \eqref{DRO_ORG}, respectively. Step-2 evaluates the lower bound for problem \eqref{DRO_ORG} by solving the relaxed master problem. Step-3 finds the worst-case value of $\tilxi^*$ by enumerating all $N$ extreme points and picks an optimal value. Step-4 updates the upper bound for the worst-case value of $\tilxi^*$. Steps 5-7 terminate the algorithm if the optimality gap is within a specified $\epsilon$, else steps 8-9 augment the master problem with the constraints and variables associated with extreme point $\tilxi^*$ and continues the iteration until an optimal solution for problem \eqref{DRO_ORG} is found.

\subsection{Triangle support set: Exact tractable reformulation}
\label{subsec:triangle}
As a special case, we further approximate the convex support set with a polytope that has \textit{three} extreme points, i.e., $\Xi^{\text{\tiny 3}}$ (see Figure \ref{fig:poly}). Once again, the CCG algorithm (algorithm \eqref{algo:CCG}) can be used to solve model \eqref{DRO_ORG}. However, with a triangle support set, we derive an exact monolithic reformulation as discussed below. This reformulation is solved efficiently using off-the-shelf commercial solvers such as CPLEX or Gurobi. In the numerical results section, we demonstrate the computational efficacy of this exact reformulation in detail. 
\begin{proposition}
	For fixed $z \in \mathcal{F}$, the worst-case expected value $\sup_{\mathbb{P} \in \mathcal{D}} \mathbb{E}_{\mathbb{P}}[ \mathcal{Q} (\zb, \tilxi) ]$ with $\Xi^{\text{\tiny 3}}$ is equivalent to
	\begin{subequations}
		\label{PrimalLP}
		\small
		\begin{align}
		\max_{p \in \mathbb{R}^3_+} \ & \sum_{k=1}^3 \mathcal{Q}(\zb, \hat{\bm{\xi}}^k) p_k \\
		\mbox{s.t.} \ 
		& \sum_{k=1}^3 \hat{\xi}_{\text{\tiny E}}^k p_k = \mu_{\text{\tiny E}}, \ 
		\sum_{k=1}^3 \hat{\xi}_{\text{\tiny N}}^k p_k = \mu_{\text{\tiny N}}, \ 
		\sum_{k=1}^3 p_k = 1.  \label{PrimalLP-3}
		\end{align}
		where $\{ \hat{\bm{\xi}}^1, \hat{\bm{\xi}}^2, \hat{\bm{\xi}}^3 \}$ are the extreme points of the set $\Xi^{\text{\tiny 3}}$.
	\end{subequations}
\end{proposition} 
\begin{proof}For fixed $z \in \mathcal{F}$, the function $\mathcal{Q}(\zb, \tilxi)$ is convex in $\tilxi$. Let $\mathcal{H}$ be a convex hull of $\{(\tilxi, y) : \tilxi \in \Xi^{\text{\tiny 3}}, \ y =\mathcal{Q}(\zb, \tilxi) \}$.
Since taking expectation can be viewed as a convex combination, it follows that $(\mub, \mathcal{Q}(\zb, \mub)) \in \mathcal{H}$.
Since the set $\mathcal{D}$ is a mean-support ambiguity set where the support set $\Xi^{\text{\tiny 3}}$ is a simplex with 3 extreme points $\hat{\bm{\xi}}^1, \hat{\bm{\xi}}^2, \hat{\bm{\xi}}^3$, there is an unique convex combination of the extreme points which yields $\mub$.
Therefore, we have $\sup_{\Pb \in \mathcal{D}} \Eb_{\Pb}[\mathcal{Q}(\zb, \tilxi)] = \sup \{ y | (\mub, y) \in \mathcal{H} \} $, where $\mub$ is a convex combination of the 3 extreme points, $\mub = \hat{\bm{\xi}}^1 p_1 + \hat{\bm{\xi}}^2 p_2 + \hat{\bm{\xi}}^3 p_3$, and $y = \mathcal{Q}(\zb, \hat{\bm{\xi}}^1) p_1 + \mathcal{Q}(\zb, \hat{\bm{\xi}}^2) p_2 + \mathcal{Q}(\zb, \hat{\bm{\xi}}^3) p_3$.
\end{proof}

\begin{remark}
	The optimal solution $(p^*_1, p^*_2, p^*_3)$ to problem \eqref{PrimalLP} is uniquely determined due to the unique convex combination of the extreme points which yields $\mub$. In other words, we solve the following linear system of equations:
	\begin{align*}
	\begin{bmatrix}
	\hat{\xi}_{\text{\tiny E}}^1 & \hat{\xi}_{\text{\tiny E}}^2 & \hat{\xi}_{\text{\tiny E}}^3 \\
	\hat{\xi}_{\text{\tiny N}}^1 & \hat{\xi}_{\text{\tiny N}}^2 & \hat{\xi}_{\text{\tiny N}}^3 \\
	1 & 1 & 1 \\		
	\end{bmatrix}
	\begin{bmatrix}
	p_1 \\ p_2 \\ p_3
	\end{bmatrix}
	=
	\begin{bmatrix}
	\mu_{\text{\tiny E}} \\ \mu_{\text{\tiny N}} \\ 1
	\end{bmatrix}
	\end{align*}
	As long as $\mub \in \Xi^{\text{\tiny 3}}$, we have $p^*_1, p^*_2, p^*_3 \geq 0$.
\end{remark}

\begin{proposition}
	Problem \eqref{DRO_ORG} with triangle support set ($\Xi^{\text{\tiny 3}}$) is equivalent to the following stochastic program:
	\begin{subequations}
	\small
	\label{SLP_Primal}		
	\begin{align}
	\min_{\zb \in \mathcal{F}} \ & \qb^{\text{\tiny T}} \zb + \sum_{k=1}^3 p^*_k ( \cb^{\text{\tiny T}} \xb^{k} ) \\
	\mbox{s.t.} \
	& \Ab \xb^k \geq \db - \Bb (\hat{\bm{\xi}}^k) \zb, \  \forall k \in \{1,2,3\}.
	\end{align}	
	\end{subequations}
\end{proposition}

\section{Numerical Experiments}
\label{sec:numerical}
In this section, we conduct numerical experiments using the epri-21 system, which is specifically designed for GMD studies. Most parameters in our models can be obtained from  \cite{horton2012test}. The remaining parameters include $\kappa^{\text{\tiny l}} = 50,000 [\$/\text{p.u.}]$, $\kappa^{\text{\tiny s}} = 100,000 [\$/p.u.]$, $\overline{v}^{\text{\tiny d}} = 10,000$ [V], and $\overline{I}^{\text{\tiny eff}}_e = 2 \frac{\overline{s}_e}{\min(\underline{v}_i, \underline{v}_j)} \text{[Amp]}, \ \forall e_{ij} \in \mathcal{E}^{\tau}$ . 
For each system, we generate instances which vary the mean values $(\mu_{\text{\tiny E}}, \mu_{\text{\tiny N}})$ and the 
maximum magnitude $R$ of the geo-electric field (all units are in [V/km]).
Computations were performed with the HPC resources at Los Alamos National Laboratory with Intel Xeon CPU E5-2660v3, 2.60GHz and 120GB of memory.  Optimization models were solved using Gurobi v8.1.0 and were implemented in \texttt{C++}.

\subsection{The epri-21 system} 
\label{sec-epri21}
Figure \ref{fig:Map_epri21} shows a simplified diagram of the epri-21 system geo-located near Atlanta, GA. This system has $19$ buses, $7$ generators, $15$ transmission lines, $16$ transformers, and $8$ sub-stations. In the diagram, the blue lines are 500kV and the green lines are 345kV (see \cite{horton2012test} for details). 
\begin{figure}[h!]
    \centering
    \includegraphics[width=\columnwidth]{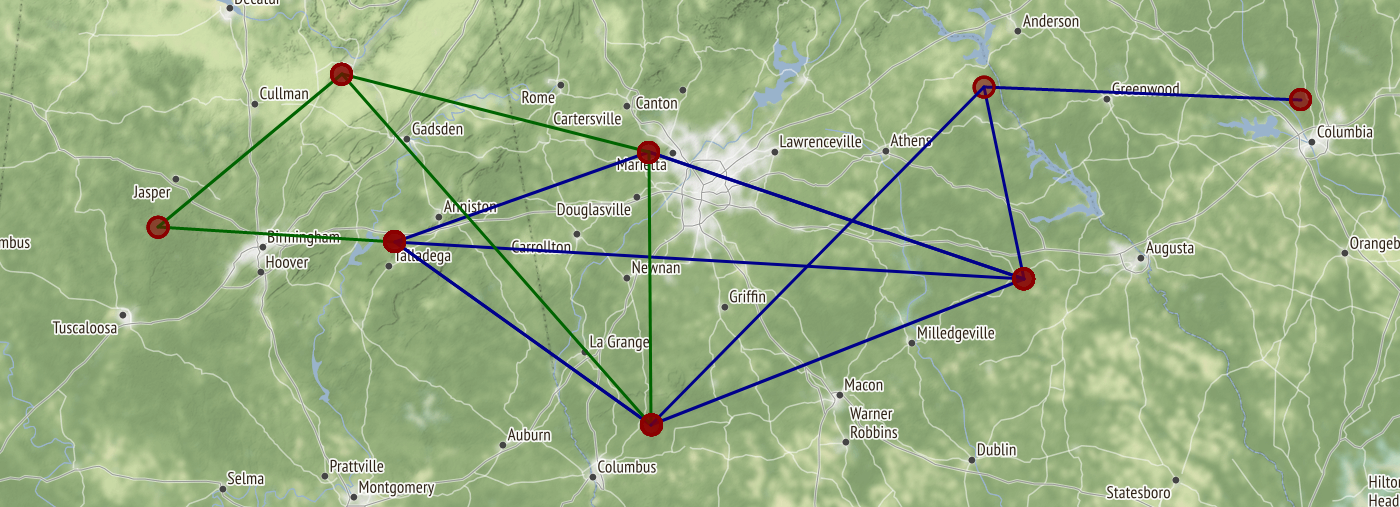}
    \caption{The epri-21 system}
    \label{fig:Map_epri21}
\end{figure}

\subsection{Uncertainty data sets} 
\label{subsec:uncertain_data}
\begin{figure}[h!]
    \centering
    \includegraphics[scale=0.265]{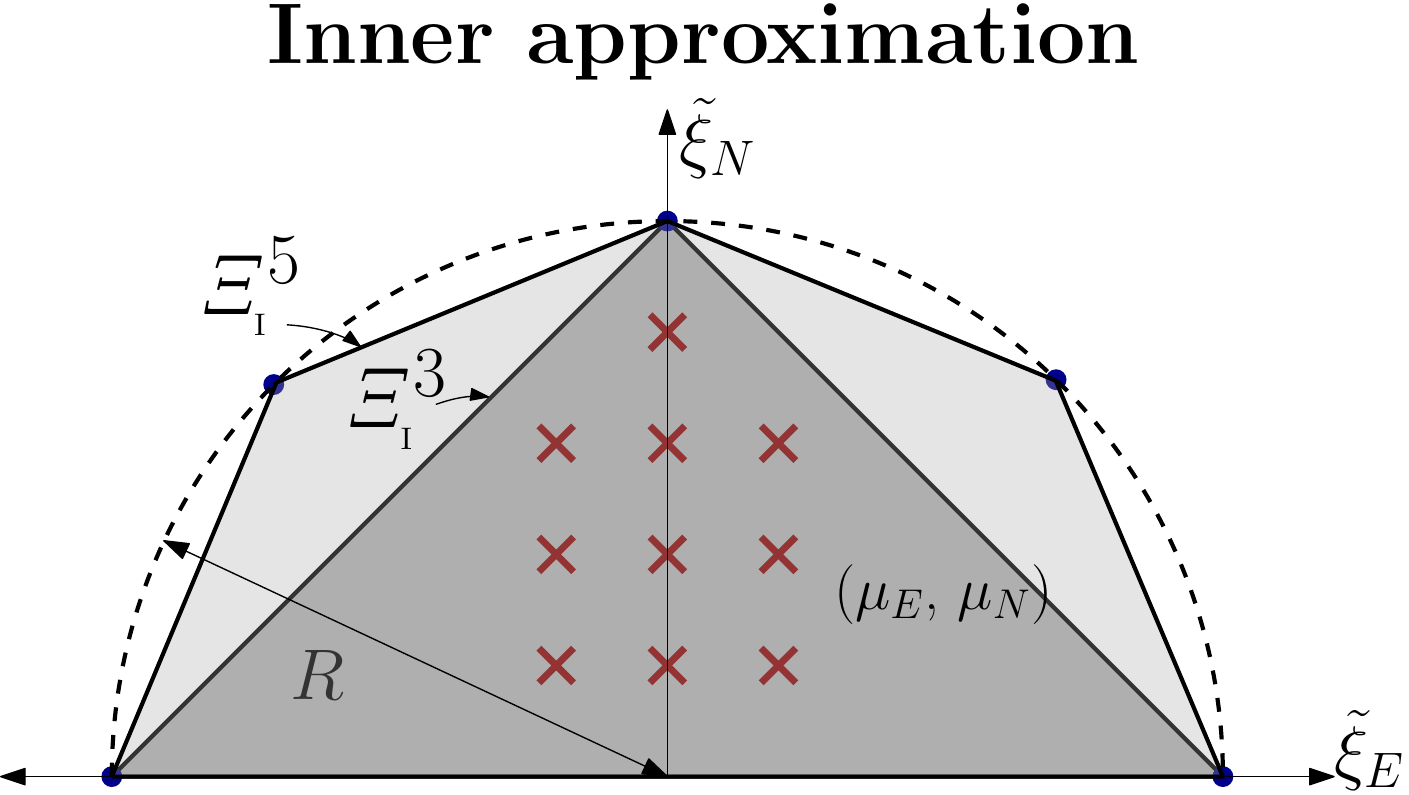}
    \includegraphics[scale=0.265]{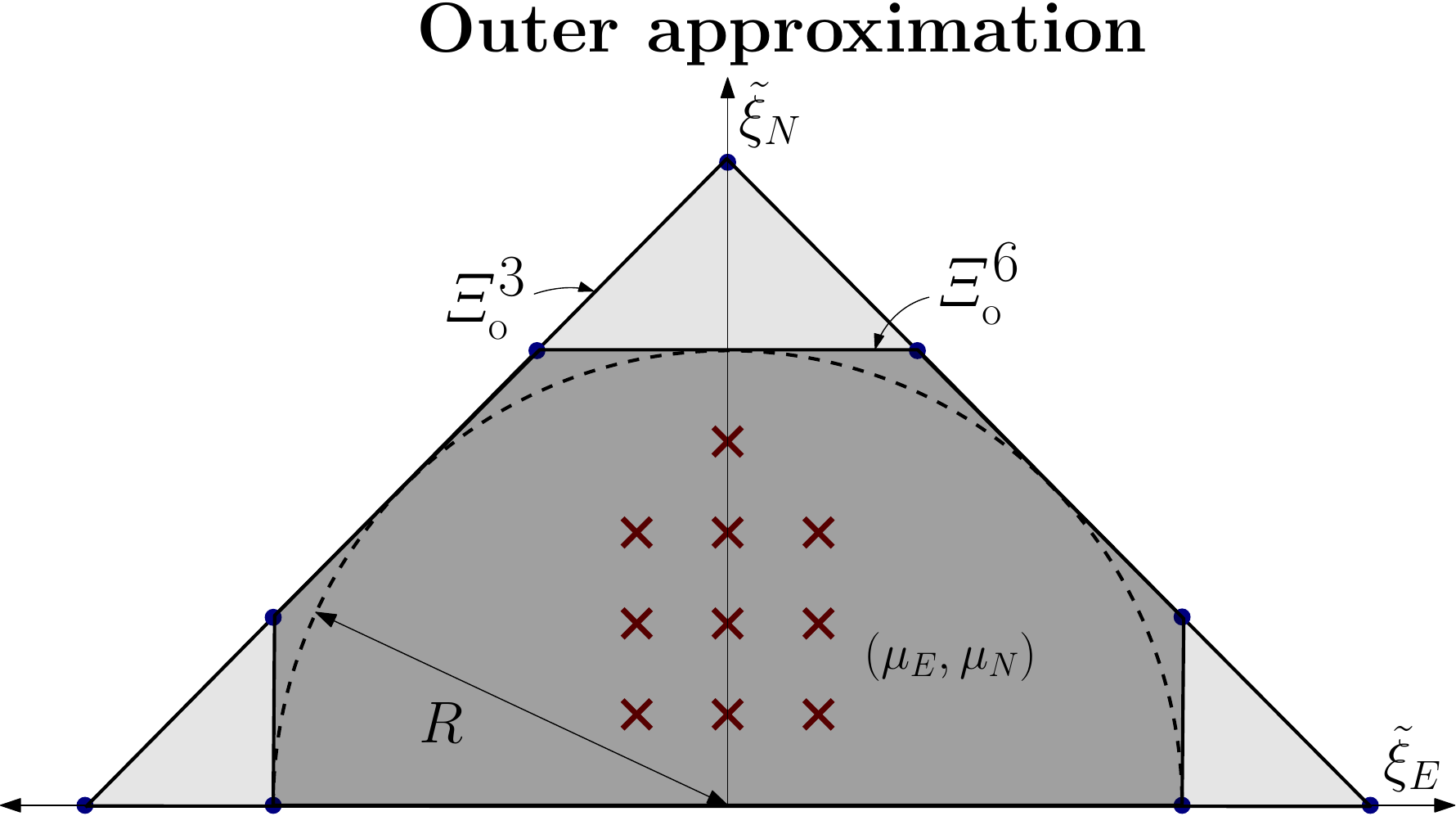}
    \caption{Polyhedral support sets that inner- (left) and outer- (right) approximate the nonlinear uncertainty set $\Xi$.}
    \label{fig:Approximate}
\end{figure}
We consider \textit{ten} different instances that vary the mean values $(\mu_{\text{\tiny E}}, \mu_{\text{\tiny N}})$
in proportion to $R$ (Table \ref{tab:10Instances}).
%, based on the radius $R$ of the uncertainty set, as described in Table \ref{tab:10Instances}.
For each instance, we construct \textit{four} different polyhedral support sets: (1) triangle ($\Xi^3_{\text{\tiny I}}$) and (2) pentagon ($\Xi^5_{\text{\tiny I}}$) that inner-approximates the convex support set and (3) triangle ($\Xi^3_{\text{\tiny O}}$) and (4) hexagon ($\Xi^6_{\text{\tiny O}}$) that outer-approximates the convex support set as depicted in Figure \ref{fig:Approximate}. The objective function values of prob. \eqref{DRO_ORG} with these sets are non-increasing in the order: $\Xi^3_{\text{\tiny O}}$, $\Xi^6_{\text{\tiny O}}$, $\Xi$, $\Xi^5_{\text{\tiny I}}$, and $\Xi^3_{\text{\tiny I}}$. 

\begin{table}[h!]
    \small
    \centering
    \begin{tabular}{|c|c|c|c|}
    \hline
    Instance \# & $(\mu_{\text{\tiny E}}, \mu_{\text{\tiny N}})$ [V/km] & Instance \# & $(\mu_{\text{\tiny E}}, \mu_{\text{\tiny N}})$ [V/km] \\ \hline
    1 & (0, R/5) & 6 & (R/5, 2R/5) \\
    2 & (0, 2R/5) & 7 & (R/5, 3R/5) \\
    3 & (0, 3R/5) & 8 & (-R/5, R/5) \\
    4 & (0, 4R/5) & 9 & (-R/5, 2R/5) \\ 
    5 & (R/5, R/5) & 10 & (-R/5, 3R/5) \\ \hline
    \end{tabular}
    \caption{Ten different $(\mu_{\text{\tiny E}}, \mu_{\text{\tiny N}})$ depending on $R$ values.}
    \label{tab:10Instances}
    % \vspace{-0.59cm}
\end{table}

\subsection{Quality of uncertainty set approximation}
\label{subsec:quality}
Since support set $\Xi^3_{\text{\tiny I}}$ has the lowest objective function value for the 4 polyhedral sets, we normalize all the objective function values with respect to that of $\Xi^3_{\text{\tiny I}}$.
Figure \ref{fig:NormalizedObjValues_epri21} shows the normalized objective function values of the DRO model for the 4 polyhedral sets on the $10$ instances that have $R=5$ (left) and $R=15$ (right). With support sets $\Xi^3_{\text{\tiny I}}$ and $\Xi^3_{\text{\tiny O}}$, the optimality gaps are less than $0.02\%$ and $0.49\%$ for $R=5$ and $R=15$, respectively.
With support sets $\Xi^5_{\text{\tiny I}}$ and $\Xi^6_{\text{\tiny O}}$, the optimality gaps are reduced to $0.01\%$ and $0.03\%$ when $R=5$ and $R=15$, respectively. Based on these optimality gaps for the epri-21 system, we observe that the triangle support set is a very good approximation of the convex support set. 

\begin{figure}[!h]
   \centering
   \includegraphics[scale=0.227]{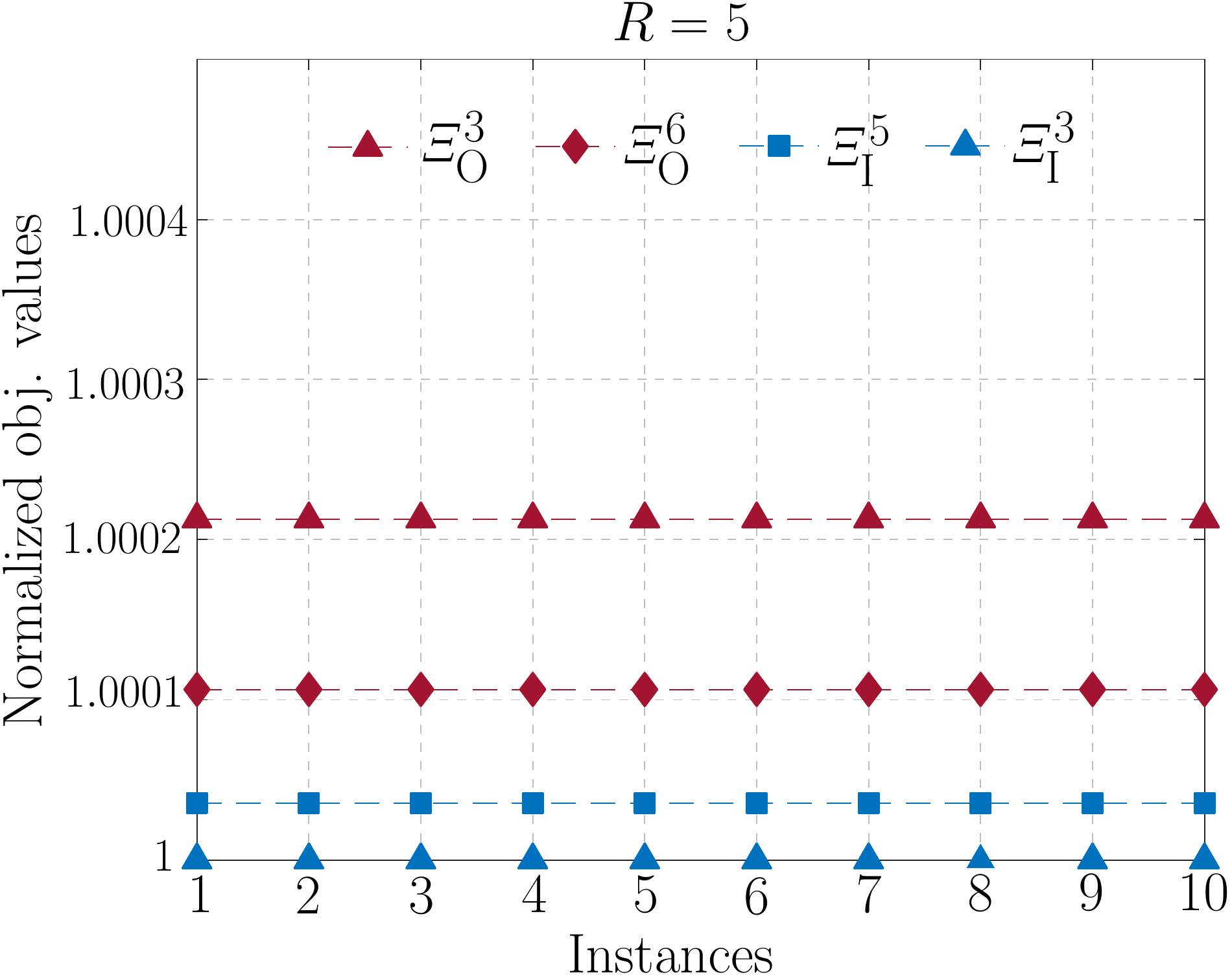}      
   \includegraphics[scale=0.227]{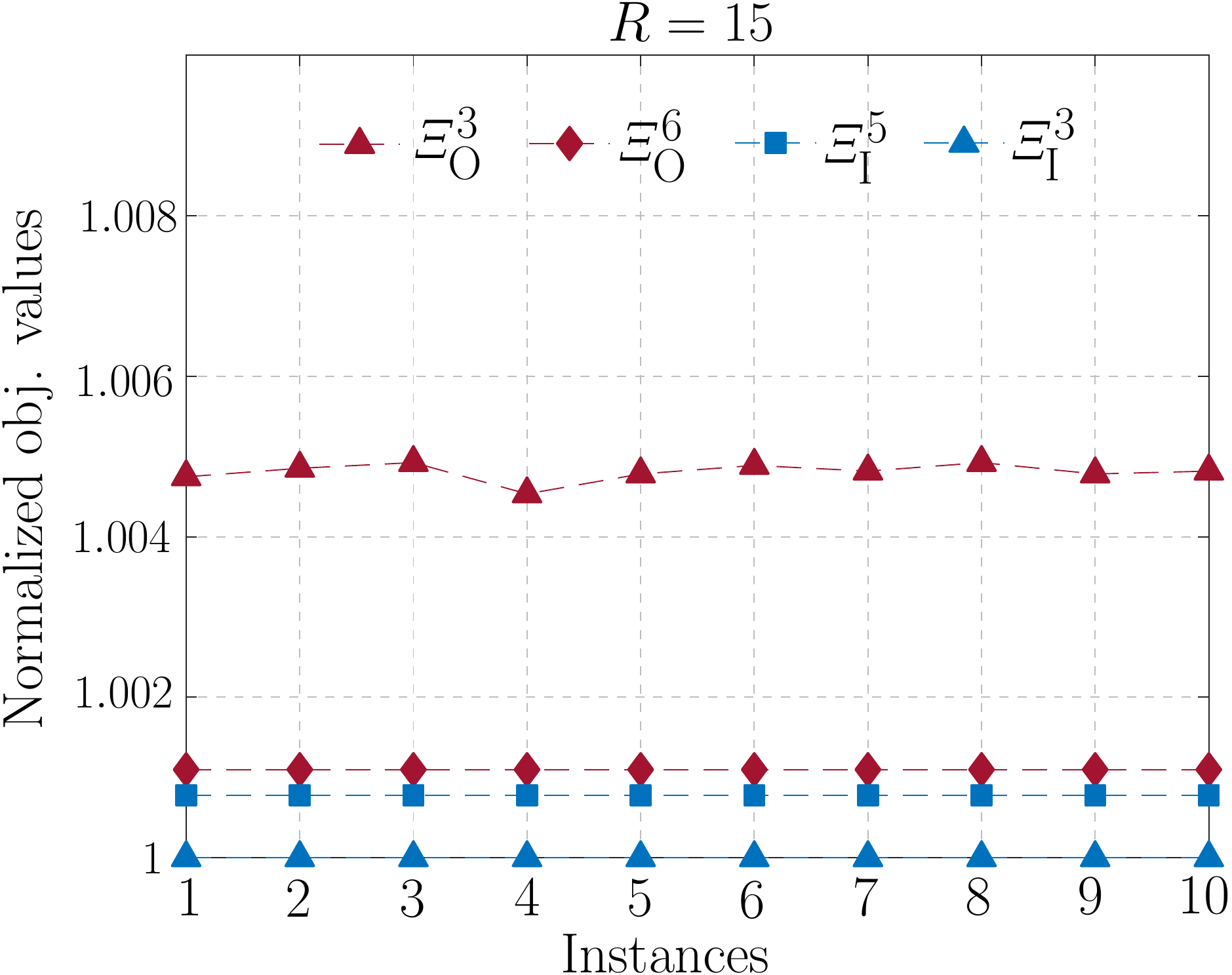}
   \caption{Normalized objective function values for $R=5,15$.} 
   \label{fig:NormalizedObjValues_epri21}
%   \vspace{-0.3cm}
\end{figure}

\subsection{Computational performances}
\label{subsec:computational}
% \vspace{-0.2cm}
\begin{figure}[!h]
    \centering
    \includegraphics[scale=0.34]{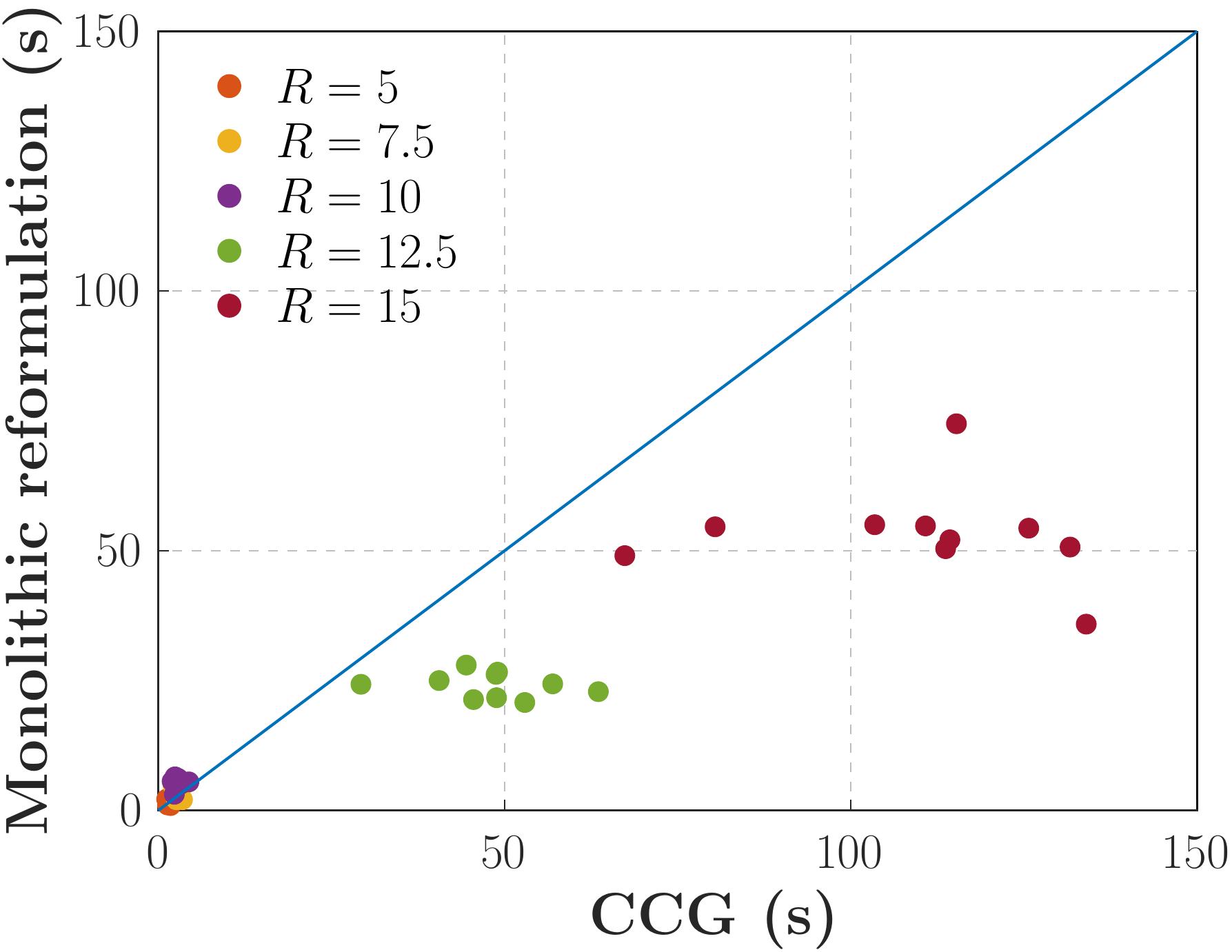}
    \caption{Computational times (sec.) of monolithic reformulation and CCG algorithm with the triangle support set ($\Xi^3_{\text{\tiny I}}$).}
    \label{fig:run_times_epri21}
\end{figure}
In this section, for $R \in \{5, 7.5, 10, 12.5, 15\}$, we construct $10$ instances with mean values described in Table \ref{tab:10Instances}. Figure \ref{fig:run_times_epri21} compares the exact monolithic reformulation (section \ref{subsec:triangle}) with the CCG algorithm (algorithm \ref{algo:CCG}) using the $\Xi^3_{\text{\tiny I}}$ formulation. 
All the points located below the 45 degree blue line are instances where the monolithic reformulation computationally outperforms the CCG algorithm. For larger uncertainty sets ($R=15$), the monolithic reformulation is on an average \textit{2.1} times faster than CCG. The tightness of the triangle-set approximation (see section \ref{subsec:quality}) and the computational efficacy of the exact reformulation provides evidence that the monolithic approach is very effective at solving the DRO.

\subsection{Planning mitigation solutions for uncertain GMDs}
\begin{figure}[!h]
  \centering
    \includegraphics[scale=0.351]{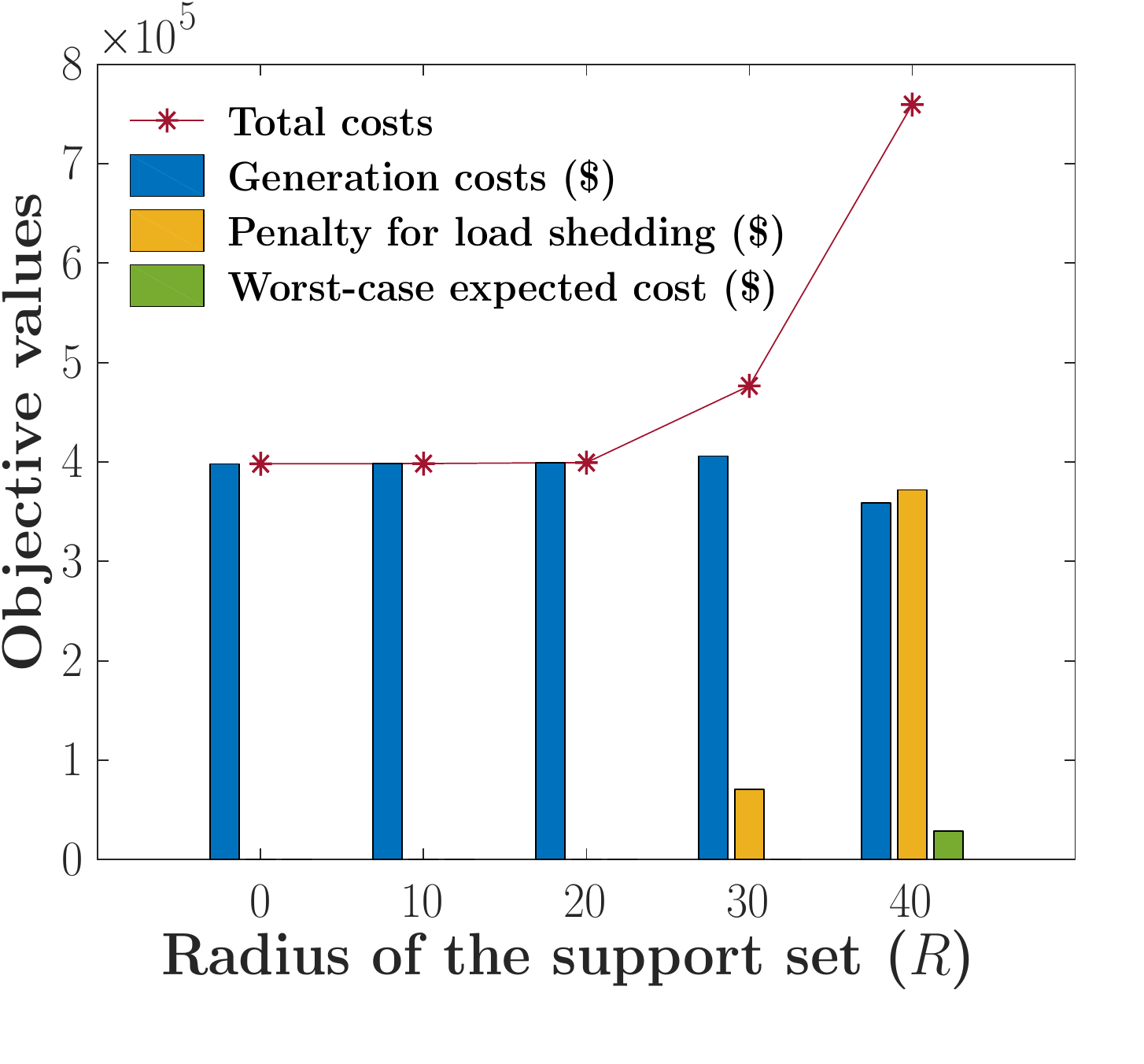}   \includegraphics[scale=0.351]{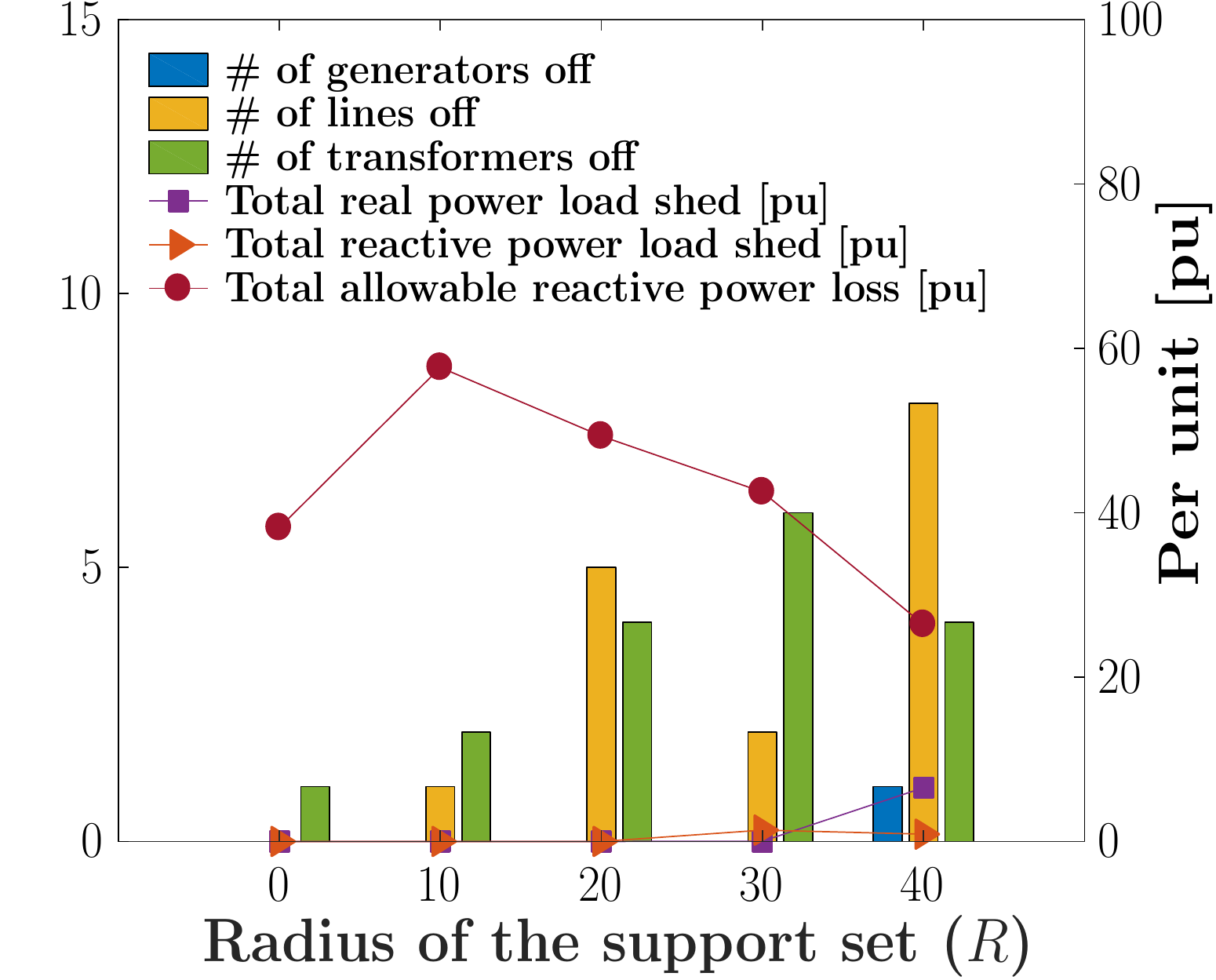} 
  \caption{Optimal objectives and solutions for various $R$ values.} 
  \label{fig:solutions_epri21}
%   \vspace{-0.3cm}
\end{figure}
%In this study, we assume that we only know the support set and the mean values of $(\tilde{\xi}_{\text{\tiny E}}, \tilde{\xi}_{\text{\tiny N}})$.
In this section, we consider how the solutions change when the mean and magnitude of the GMD are varied.
Here, the mean values are $(\mu_{\text{\tiny E}}, \mu_{\text{\tiny N}})=(5,4)$ and $R \in \{0, 10, 20, 30, 40\}$. $R=0$ indicates a deterministic model with $(\tilde{\xi}_{\text{\tiny E}}, \tilde{\xi}_{\text{\tiny N}})=(\mu_{\text{\tiny E}}, \mu_{\text{\tiny N}})=(5,4)$.

In Figure \ref{fig:solutions_epri21}, the optimal objective values (top) and mitigation actions (bottom) change as the size of the uncertainty set ($R$) increases. 
As expected, as $R$ increases, more generators, lines, and transformers are turned off to mitigate the effect of GMDs and address a larger number of worst-case scenarios. The negative impacts of GMDs can be mitigated without shedding loads (only using line and transformer switching) for $R\le 20$. 
However, when $R \geq 10$, the total allowable reactive power losses (\eqref{eq:stg1}'s solution, bottom figure) decreases, indicating that the topological control actions can mitigate these losses, thus potentially reducing the need for expensive blocking devices used in the network. When $R\geq 30$, the topological control actions are not sufficient to handle the uncertainty in the GMD and some real and reactive power loads are shed. This is reflected in the increase in the total cost of the objective value at $R=40$.

\section{Conclusions}
In this paper, we developed a novel two-stage DRO model which uses control of transmission lines, generators, and transformers to mitigate potential negative impacts of uncertain GMDs.
This model minimizes the expected total cost of mitigation for the worst-case distribution in a convex support set of the GMD's uncertainty and subject to convex relaxations of the AC power flow and GMD constraints.
Given this convex support set and mean values for the uncertain GMDs, our DRO model is solvable using the CCG algorithm. However, there are no guarantees of finite time convergence. Instead, we approximated the support set with a polytope with $N$ extreme points that allows the CCG to terminate with a finite number of iterations ($O(N)$). We further reformulated the two-stage DRO model into a monolithic MISOCP for the special case when the support set contains three extreme points. We numerically showed the run-time efficacy of this reformulation. Finally, we provided a detailed case study on epri-21 system which analyzed the effects of modeling uncertain GMDs.

There are a number of interesting future directions for this work. \textit{First}, given the tightness of the triangle support set approximation and the computational efficacy of the exact reformulation, this approach could be used to warm-start the CCG algorithm and speed up the convergence in cases with $N$ extreme points. \textit{Second}, the approximation could be tightened further by considering different choices of the extreme points for the triangle. \textit{Third}, the corrective actions obtained in this paper may not necessarily be feasible to the original nonconvex AC power flow and GMD constraints. Thus AC feasible solution recovery will be important from the practical perspective, albeit this may be non-trivial for the DRO version of the problem. \textit{Finally}, it will be important to scale the DRO to cases with 100's or even 1000's of nodes. \\

% \section*{Acknowledgements}
\noindent
\textbf{Acknowledgements} This  work  was supported by the U.S. Department of Energy LDRD program at Los Alamos National Laboratory under \emph{``Impacts of Extreme Space Weather Events on Power Grid Infrastructure: Physics-Based Modelling of Geomagnetically-Induced Currents (GICs) During Carrington-Class Geomagnetic Storms"}.
%  Los Alamos National Laboratory is operated by Triad National Security, LLC, for the National Nuclear Security Administration of U.S. Department of Energy (Contract No. 89233218CNA000001).
% under
% the auspices of the NNSA of the U.S. DOE at LANL under Contract No.  DE-
% AC52-06NA25396.

\bibliographystyle{IEEEtran}
\bibliography{references.bib}

% that's all folks
\end{document}